\documentclass[12pt,a4paper]{amsart}
 \pdfoutput=1
\usepackage[a4paper,inner=2.5cm,outer=2.5cm,top=2.5cm,bottom=2.5cm]{geometry}
\usepackage{amsmath,amssymb,amsthm,enumerate
}
\usepackage{hyperref}
\hypersetup{colorlinks=true,linkcolor=blue,citecolor=purple,filecolor=magenta,urlcolor=cyan}

\usepackage{float}

\usepackage{extarrows}

\usepackage{xcolor}

\newcommand{\C}{\mathbb{C}}

\theoremstyle{plain}
\newtheorem{theorem}{Theorem}[section]
\newtheorem{proposition}[theorem]{Proposition}
\newtheorem{corollary}[theorem]{Corollary}
\newtheorem{lemma}[theorem]{Lemma}
\theoremstyle{definition}
\newtheorem{definition}[theorem]{Definition}

\newtheorem{notation}[theorem]{Notation}
\theoremstyle{remark}
\newtheorem{remark}[theorem]{Remark}

\newcommand{\Z}{\mathbb{Z}}
\newcommand{\cD}{\mathcal{D}}
\newcommand{\cS}{\mathcal{S}}

\newcommand{\cE}{\mathcal{E}}
\newcommand{\cF}{\mathcal{F}}

\newcommand{\VEV}[1]{{\big\langle 0 \big| {#1} \big| 0 \big\rangle}}

\newcommand{\res}{\mathop{\rm res}}

\newcommand{\Hd}{H^\bullet}
\newcommand{\Hc}{H}
\newcommand{\euro}{\mathbb{J}}

\begin{document}

\title{Explicit closed algebraic formulas for 
Orlov--Scherbin $n$-point functions}

\author[B.~Bychkov]{Boris~Bychkov}
\address{B.~B.: Faculty of Mathematics, National Research University Higher School of Economics, Usacheva 6, 119048 Moscow, Russia; and Center of Integrable Systems, P.G. Demidov Yaroslavl State University, Sovetskaya 14, 150003,Yaroslavl, Russia}
\email{bbychkov@hse.ru}

\author[P.~Dunin-Barkowski]{Petr~Dunin-Barkowski}
\address{P.~D.-B.: Faculty of Mathematics, National Research University Higher School of Economics, Usacheva 6, 119048 Moscow, Russia; HSE--Skoltech International Laboratory of Representation Theory and Mathematical Physics, Skoltech, Nobelya 1, 143026, Moscow, Russia; and ITEP, 117218 Moscow, Russia}
\email{ptdunin@hse.ru}

\author[M.~Kazarian]{Maxim~Kazarian}
\address{M.~K.: Faculty of Mathematics, National Research University Higher School of Economics, Usacheva 6, 119048 Moscow, Russia; and Center for Advanced Studies, Skoltech, Nobelya 1, 143026, Moscow, Russia}
\email{kazarian@mccme.ru}

\author[S.~Shadrin]{Sergey~Shadrin}
\address{S.~S.: Korteweg-de Vries Institute for Mathematics, University of Amsterdam, Postbus 94248, 1090 GE Amsterdam, The Netherlands}
\email{S.Shadrin@uva.nl}	

\begin{abstract}
	We derive a new explicit formula in terms of sums over graphs for the $n$-point correlation functions of general formal weighted double Hurwitz numbers coming from the Kadomtsev--Petviashvili tau functions of hypergeometric type (also known as Orlov--Scherbin partition functions). Notably, we use the change of variables suggested by the associated spectral curve, and our formula turns out to be a polynomial expression in a certain small set of formal functions defined on the spectral curve.
\end{abstract}

\maketitle	
	
\tableofcontents

\section{Introduction}

\subsection{Hurwitz numbers and KP tau functions of hypergeometric type}

Hurwitz numbers enumerate topologically distinct ramified coverings of the sphere $S^2$ by Riemann surfaces with prescribed ramification data. Different types of Hurwitz numbers are distinguished by the way the ramification data is specified. This data can be encoded in the values of parameters $c_k$, $s_k$, $k=1,2,\dots$, collected into two formal power series

\begin{align}\label{eq:psiY}
\psi(y)&:=\sum_{k=1}^\infty c_k y^k,\\
y(z)&:=\sum_{k=1}^\infty s_k z^k.
\end{align}
We do not reproduce here the precise combinatorial definition of the Hurwitz numbers we are interested in, instead, we identify them as the Taylor coefficients of the corresponding generating function~$F(p_1,p_2,\dots)$ introduced below. Namely, its exponential $Z=\exp F$ is a \emph{Kadomtsev--Petviashvili tau function of hypergeometric type} (also known as an \emph{Orlov--Scherbin partition function})~\cite{Kharchev,OrlovScherbin,OrlovScherbin-2} given explicitly by its expansion in the basis of Schur functions
\begin{equation}\label{eq:OS}
Z=e^{F}=\sum_\lambda e^{\sum_{(i,j)\in\lambda}\psi(\hbar(j-i))}s_\lambda(p)s_\lambda(s/\hbar).
\end{equation}
We regard $Z$ and $F$ as formal power series in the variables $p_1,p_2,\dots$ depending on additional parameters $c_k, s_k$, and~$\hbar$.
The summation runs over the set of all partitions (Young diagrams)~$\lambda$ including the empty one, $s_\lambda$ denotes the corresponding Schur symmetric function represented as a polynomial in the power sums~$p_k$. Parameters $c_k$ are involved as the coefficients of the series~$\psi$ while $s_k$ are substituted as the arguments of $s_\lambda$ via $s/\hbar=(s_1/\hbar,s_2/\hbar,\dots)$. We regard a Young diagram $\lambda$ as a table of rows of lengths $\lambda_1\ge\lambda_2\ge\dots\ge0$, and for a cell of this table with coordinates $(i,j)$ its \emph{content} is defined as the difference $j-i$ of coordinates. By that reason, the exponent $e^{\sum_{(i,j)\in\lambda}\psi(\hbar(j-i))}$ is referred to sometimes as \emph{content product}.

The (formal) Hurwitz numbers $h_{g,(m_1,\dots,m_n)}$ associated with the series $F$ are defined by the expansion
\begin{equation}
\frac{\partial^n F}{\partial p_{m_1}\dots\partial p_{m_n}}\Bigm|_{p=0}=
\sum_{g=0}^\infty\hbar^{2g-2+n}h_{g,(m_1,\dots,m_n)}.
\end{equation}

Generating functions for many particular families of Hurwitz numbers (e.g. simple, monotone, Bousquet-M\'elou--Schaeffer numbers, Grothendieck's dessins d'enfants, and many others numbers of similar nature both of single, orbifold or double types) are included in~$F$ for particular values of parameters, see~Table \ref{tab:Hurw} (cf.~\cite{ACEH-1,ALS,Harnad2014,KazLand}).
\begin{table}[H]
	\renewcommand{\arraystretch}{1.5}
	\begin{tabular}{|c|c|}
		\hline
		Hurwitz numbers & $e^{\psi(y)}$\\
		\hline
		\hline
		usual & $e^{y}$ \\
		atlantes
		& $e^{y^r}$ \\
		monotone & $1/(1-y)$ \\
		strictly monotone & $1+y$ \\
		hypermaps & $(1+uy)(1+vy)$\\
		BMS numbers & $(1+y)^m$ \\
		polynomial weighted& $\sum_{k=1}^d c_k y^k$\\
		\hline
		general weighted & $\exp\left({\sum_{k=1}^\infty c_k y^k}\right)$\\
		\hline
	\end{tabular} \quad \begin{tabular}{|c|c|}
	\hline
	Variations & $y(z)$\\
	\hline
	\hline
	simple & $z$ \\
	orbifold & $z^q$ \\
	\hline
	double & $\sum_{k=1}^\infty s_kz^k$ \\
	\hline
\end{tabular}
	\renewcommand{\arraystretch}{1}
	\caption{Types of Hurwitz numbers}
	\label{tab:Hurw}
\end{table}

In the most general case, when~$\psi$ and~$y$ are arbitrary power series, the Taylor coefficients~$h_{g,(m_1,\dots,m_n)}$ have combinatorial meaning of \emph{weighted double Hurwitz numbers} (see e.~g.~\cite{Harnad2014}).
Roughly speaking, when regarding $F$ as a generating series for Hurwitz numbers, the ramification over the point $\infty\in S^2=\C P^1$ is encoded by a monomial in $p$-variables, the ramification over $0$ corresponds to a monomial in $s$-variables, the ramification types over the points different from~$0$ and~$\infty$ are described by the explicit form of the series~$\psi$. The exponent of the variable $\hbar$ is the negative Euler characteristic $2g-2+n$ of the covering surface punctured at the preimages of~$\infty$ where~$g$ is the genus of the covering surface and~$n$ is the number of preimages of~$\infty$.

\subsection{\texorpdfstring{$n$}{n}-point functions}

Formula~\eqref{eq:OS} is quite explicit and efficient for the numerical computation of particular Hurwitz numbers. Therefore, the main interest is related not to the computation of a single Hurwitz number but to the study of analytical and integrable properties of their generating functions. These properties are often formulated in terms of the \emph{connected} and \emph{disconnected}, respectively, \emph{$n$-point correlation functions} defined by
\begin{align}\label{eq:Hndef}
H_n&=\sum_{k_1,\dots,k_n=1}^\infty\frac{\partial^n F}{\partial p_{k_1}\dots\partial p_{k_n}}\Bigm|_{p=0}X_1^{k_1}\dots X_n^{k_n},\\ \label{eq:Hdndef}
\Hd_n&=\sum_{k_1,\dots,k_n=1}^\infty\frac{\partial^n Z}{\partial p_{k_1}\dots\partial p_{k_n}}\Bigm|_{p=0}X_1^{k_1}\dots X_n^{k_n}.
\end{align}
These are infinite power series in $X_1,\dots,X_n$ serving as an alternative way of collecting Hurwitz numbers enumerating connected and disconnected, respectively, coverings of the sphere. Connected and disconnected $n$-point functions are related to one another by inclusion-exclusion relations
\begin{eqnarray}\label{eq:inclexclH}
\Hd_n(X_{\{1,\dots,n\}})= \sum\limits_{I\vdash\{1,\ldots,n\}}\prod\limits_{i=1}^{|I|}H_{|I_i|}(X_{I_i}),\\
H_n(X_{\{1,\dots,n\}})= \sum\limits_{I\vdash\{1,\ldots,n\}}(-1)^{|I|-1}(|I|-1)!\prod\limits_{i=1}^{|I|}\Hd_{|I_i|}(X_{I_i}),
\end{eqnarray}
where the sums run over all unordered partitions of the set $\{1,\dots,n\}$, and for $I=\{i_1,i_2,\dots\}$ we denote $X_I:=(X_{i_1},X_{i_2},\dots)$. The connected $n$-point function admits a \emph{genus decomposition}
\begin{equation}
H_n=\sum_{g=0}^\infty \hbar^{2g-2+n}H_{g,n},
\end{equation}
where $H_{g,n}$ is independent of~$\hbar$:
\begin{equation}
H_{g,n}=\sum_{m_1,\dots,m_n=1}^\infty h_{g,(m_1,\dots,m_n)}X_1^{m_1}\dots X_n^{m_n}.
\end{equation}
This follows from combinatorial interpretation of Hurwitz numbers, but it is also a formal corollary of the computation of~$H_{n}$ of the present paper.

One of the main discoveries of last years in the theory of Hurwitz numbers is the fact that in many cases the $n$-point functions $H_{g,n}$ are governed by the \emph{topological recursion}, a formalism allowing to compute $H_{g,n}$ inductively in~$g$ and~$n$. One of the most general cases for which the topological recursion relations had been proved by the time we wrote the first version of the present paper is the one when both $e^{\psi(y)}$ and $y(z)$ are polynomials~\cite{ACEH-1,ACEH-2}. Consider the following power series
\begin{equation}\label{eq:X(z)}
X(z)=z\,e^{-\psi(y(z))}
\end{equation}
where $\psi$ and $y$ are given by~\eqref{eq:psiY}, and apply the local change of coordinates $X_i=X(z_i)$ to each of the arguments of $H_{g,n}$. One of the corollaries of the topological recursion is that the \emph{function $H_{g,n}$ written in $z$-coordinates is rational}. Do note, however, that for the approach of \cite{ACEH-1,ACEH-2} the polynomiality of $e^{\psi(y)}$ and $y(z)$ was crucial. See also Remark~\ref{rem:toprec} below. 

In this paper we show that the function $H_{g,n}$ simplifies considerably after the change~\eqref{eq:X(z)} even without any assumption of polynomiality or rationality (or even convergence) for the series $\psi(y)$ and~$y(z)$. The main result of the paper is an explicit closed formula for $H_{g,n}$ for each pair $(g,n)$. Through the change~\eqref{eq:X(z)} we have
\begin{align}\label{eq:Ddef}
D&:= X\frac{\partial}{\partial X}=\frac{1}{Q}z\frac{\partial}{\partial z},
\end{align}
where
\begin{align}\label{eq:Qdef}
Q&=\frac{z}{X}\frac{dX}{dz}=1-D\psi(y(z))=1-z\,y'(z)\,\psi'(y(z)).
\end{align}

\begin{theorem}\label{th:intro}
In the unstable cases $2g-2+n\le0$ the $n$-point functions are given by
\begin{align}
D_1 H_{0,1}&=y(z_1),\label{eq:H01}\\
H_{0,2}&=\log\left(\frac{z_1^{-1}-z_2^{-1}}{X_1^{-1}-X_2^{-1}}\right),\label{eq:H02}
\end{align}
and for all $(g,n)$ with $2g-2+n>0$ the function $H_{g,n}$ written in $z$-coordinates admits a closed expression of the form
\begin{equation}\label{eqTh1}
H_{g,n}=\sum_{j_1,\dots,j_n = 0}^\infty D_1^{j_1}\dots D_n^{j_n}
 \frac{P_{g;j_1,\dots,j_n}}{Q_1\dots Q_n}+c_{g,n}
\end{equation}
with finitely many nonzero summands, where $Q_i=Q(z_i)$, $D_i=D(z_i)=\frac1{Q_i}z_i\partial_{z_i}$, and $P_{g;j_1,\dots,j_n}$ is a polynomial combination of functions 
$\frac{z_j}{z_i-z_j}$ and 
derivatives $\psi^{(k)}(y(z_i))$ and
$\left(z_i\partial_{z_i}\right)^ky(z_i)$, $k\ge1$, $i=1,\dots, n$.
Finally, $c_{g,n}$ is a constant explicitly given by
\begin{equation}
c_{g,n}=(-1)^n\;\psi^{(2g-2+n)}(0)\;[u^{2g}]\left(\frac{u}{e^{u/2}-e^{-u/2}}\right)^2,
\end{equation}
where $[u^{2g}]$ denotes the coefficient in front of $u^{2g}$ in the series expansion.
\end{theorem}

In particular, an immediate corollary of this theorem is the following statement:

\begin{corollary} If both $y'(z)$ and $\psi'(y)$ are rational functions then $H_{g,n}$ is a rational function in $z_1,\dots,z_n$.
\end{corollary}

An explicit description of the terms entering the formula for $H_{g,n}$ (i.e. a formula where all polynomials $P_{g;j_1,\dots,j_n}$ are given explicitly) is presented in Theorem~\ref{th:MainTh} in the case $n>2$ and in Section~\ref{sec:exceptions} in the exceptional cases $n=1$ and $n=2$. It might look somewhat wired but it is actually quite explicit and can be used for practical computations. The formula holds true even in those cases when $\psi(y)$ and~$y(z)$ are just formal series with no assumption of rationality or convergence and the topological recursion is not applicable in principle. Moreover, even in those cases when $e^{\psi(y)}$ and~$y(z)$ are such that 
the topological recursion can be applied (e.g. when they are polynomial, as in \cite{ACEH-1,ACEH-2}, or in a more general case as referred to in Remark~\ref{rem:toprec}) our formula is more efficient since the number of its terms does not depend on the degrees of those polynomials and it does not require finding roots of algebraic equations determining critical points of the function~$X(z)$.


\begin{remark}
The left hand side of~\eqref{eqTh1} is a formal power series in~$z_1,\dots,z_n$ while individual summands of the right hand side have poles on the diagonals $z_i=z_j$ and their interpretation requires additional comments. First note that if both $\psi'(y)$ and $y'(z)$ are rational functions then all terms of~\eqref{eqTh1} are also rational, and the equality implies, in particular, that all poles on the diagonals on the right hand side cancel out (see Corollary~\ref{cor:Hdiag}).

In the general case, one of the possibilities to interpret Equation~\eqref{eqTh1} is to consider asymptotic Laurent expansion of all of its terms in the sector $|z_1|\ll  |z_2|\ll\dots\ll |z_n|\ll 1$. This power expansion involves monomials in $z_1,\dots, z_n$ containing both positive and negative powers of the variables $z_i$.

It is much more advisable, however, to treat the terms of~\eqref{eqTh1} in a different way. Namely we consider them as elements of the  ring $R=\mathbb{C}[[z_1,\dots,z_n]][\{(z_i-z_j)^{-1};i,j=1,\ldots,n\}]$ of `formal power series with finite order poles on the diagonals'. It follows that for each $d\ge0$ the term of homogeneous degree~$d$ of each summand in~\eqref{eqTh1} is expressed as a degree~$d$ homogeneous rational function in $z_1,\dots,z_n$ with possible poles on the diagonals. After summation, all these poles cancel out and the result is a homogeneous polynomial representing degree~$d$ homogeneous term of the Taylor expansion of $H_{g,n}$.

In this paper, we first deal with formal series in $z_1,\dots, z_n$ (from definitions \eqref{eq:Hndef}--\eqref{eq:Hdndef}, where we substitute $X_i$ with $X(z_i)$ from \eqref{eq:X(z)}, itself understood as a formal series in $z_i$). Then, starting with Proposition \ref{prop:HdU}, we introduce functions $z_iz_j/(z_i-z_j)^2$ understood as their Laurent expansions in the sector $|z_1|\ll  |z_2|\ll\dots\ll |z_n|\ll 1$. Finally, in Proposition \ref{prop:mainprop} and in what follows after it, we understand all terms as elements of the ring $R$ (which is not possible to do earlier).
\end{remark}

\subsection{Further remarks}

\begin{remark}
	Our results can be naturally extended to the case where $\psi(y)$ and $y(z)$ depend on $\hbar^2$, i.e. where $c_k$ and $s_k$ are formal series in $\hbar^2$ rather than just constants. This is done in~\cite{BDKS21toprec}. See also Remarks \ref{rem:DHdeformed} and \ref{rem:Hdeformed}. This means that our statement, in addition to the cases listed in Table~\ref{tab:Hurw}, also covers e.~g.~the cases of $r$-spin 
	Hurwitz numbers~\cite{KLPS19} and the coefficients of the extended Ooguri--Vafa partition functions of colored HOMFLY polynomials of torus knots~\cite{DPSS,DKPSS}; see Table~\ref{tab:Hurwext}, which is an extension of Table~\ref{tab:Hurw} to these cases.	
	\begin{table}[H]
		\renewcommand{\arraystretch}{3}
		\begin{tabular}{|c|c|c|}
			\hline
			Hurwitz numbers & $e^{\psi(y)}$ & $y(z)$\\
			\hline
			\hline
			$r$-spin $q$-orbifold & $\exp\left(\dfrac{(y+\hbar/2)^{r+1}-(y-\hbar/2)^{r+1}}{(r+1)\;\hbar}\right)$ & $z^q$\\
			\hline
			ext. Ooguri-Vafa
			& $e^{\frac{P}{Q}y}$ & $\sum_{k=1}^\infty \dfrac{\hbar\,(A^k-A^{-k})\, z^k}{
			e^{k\hbar/2}-e^{-k\hbar/2}
			}$\\
			\hline
		\end{tabular}
		\renewcommand{\arraystretch}{1}
		\caption{Types of Hurwitz-like numbers requiring $\hbar$-extension}
		\label{tab:Hurwext}
	\end{table}
\end{remark}

\begin{remark}\label{rem:combalg}
Note that for usual simple Hurwitz numbers \cite{DKOSS,KLS}, for orbifold Hurwitz numbers \cite{DLPS,KLS}, for monotone and strictly monotone orbifold Hurwitz numbers \cite{KLS}, for $r$-spin (and $r$-spin orbifold) Hurwitz numbers \cite{KLPS19}, for the numbers of maps and hypermaps (dessins d'enfants) \cite{KZ2015}, for the Bousquet-M\'elou--Schaeffer numbers \cite{BDS}, for the coefficients of the extended Ooguri-Vafa partition function of the colored HOMFLY polynomials of torus knots~\cite{DPSS}, and for double Hurwitz numbers~\cite{BDKLM} there exist combinatorial-algebraic proofs of the so-called quasi-polinomiality property. This property, in particular, implies the linear loop equation and the projection property of \cite{BS17} for the respective $n$-point functions. We remark that the results of the present paper, in particular, serve as an independent proof of linear loop equations for all these cases (and, indeed, in the whole generality of the formal weighted double Hurwitz numbers context). We discuss this in more detail in our subsequent publication~\cite{BDKS21toprec}.
\end{remark}

\begin{remark} One way to interpret the statements of Theorem~\ref{th:intro} and Theorem~\ref{th:MainTh} is to say that they give a conceptual explanation why the change of variables~\eqref{eq:X(z)} is so ubiquitous in the weighted Hurwitz theory. This change of variables was suggested by Alexandrov--Chapuy--Eynard--Harnad in~\cite{ACEH-1} based on the explicit computation of $H_{0,1}$ and the idea that the $(0,1)$-function should determine the spectral curve for the topological recursion, in the cases when the spectral curve topological recursion is applicable. But the question why this change of variables is useful and natural for higher $H_{g,n}$ remained open until the present paper (despite some partial answers given in~\cite{ACEHferm} and in the combinatorial-algebraic papers mentioned in Remark~\ref{rem:combalg}).
\end{remark}

\begin{remark}\label{rem:toprec}
	The results of the present paper have very strong corollaries for the theory of topological recursion for various types of Hurwitz numbers, including all the ones mentioned in Remark~\ref{rem:combalg}. Specifically, in our subsequent paper~\cite{BDKS21toprec}, based on the results of the present paper, we prove the \emph{blobbed} topological recursion (defined in~\cite{BS17}) for generalized weighted double Hurwitz numbers basically in full generality, and we prove the regular topological recursion for two very general families of generalized weighted double Hurwitz numbers. These families include as special cases all the cases of Hurwitz-type numbers for which topological recursion was known from the literature (in particular, all the ones mentioned in Remark~\ref{rem:combalg}), and are actually quite a bit more general than that. Importantly, while previously in the literature the topological recursion for various types of Hurwitz-like numbers has been proved on a case-by-case basis with complicated techniques which differed between the cases, our technique of~\cite{BDKS21toprec} (based on the results of the present paper) gives a clear and uniform way to do this and highlights the underlying common structure.
	
	Moreover, the results of the present paper are also applicable beyond Hurwitz numbers. In particular, we applied them for \emph{maps} and \emph{stuffed maps} and their generalizations: in our another subsequent paper~\cite{BDKS21fully}, based on the results of the present paper, we prove a general duality for the generalized stuffed maps which we call the \emph{ordinary vs fully simple} duality, which also allowed us in that same paper to prove the Borot--Garcia-Failde conjecture on the topological recursion for fully simple maps.
	
\end{remark}

\subsection{Prior work of the third named author} The main result of this paper resolves a slightly weaker conjecture of the third named author that he posed in various talks in 2019, see e.~g.~\cite{Kaz2019}. Namely, he conjectured the existence of universal formulas for the Orlov--Scherbin $n$-point functions $H_{g,n}$ which should represent them as expressions polynomial in 
\begin{align}
& \psi^{(j)}(y(z_k)),  & j\geq 1,\  k=1,\dots, n, \\
& (z_k\partial_{z_k})^j y(z_k),  & j\geq 1,\  k=1,\dots, n, \\
& z_\ell/(z_k-z_\ell),  & 1\leq k < \ell \leq n, \\
& Q(z_k)^{-1},  & k=1,\dots,n
\end{align}
(cf. the statement of Theorem~\ref{th:intro}).
Moreover, using a variety of deformation techniques he later proved his conjecture in~\cite{Kaz2020}, and his proof gave an algorithm to produce the universal formulas inductively (see also~\cite{Kaz2020a}).

It is important to stress that although this paper resolves the conjecture of the third named author in a different way than in~\cite{Kaz2020}, and the formulas for $H_{g,n}$ given in Theorem~\ref{th:MainTh} have closed form (as opposed to their inductive algorithmic derivation in~\cite{Kaz2020}), the present paper is both ideologically and technically very much dependent on~\cite{Kaz2020}. In particular, many lemmata and computational ideas that we use below are shared directly from~\cite{Kaz2020}.

\subsection{Organization of the paper} In Section~\ref{sec:OperFockSpace} we recall the basic formalism of the operators on the bosonic Fock space that we use throughout the paper. In Section~\ref{sec:preliminary} we compute $H_{g,n}$ as a series in $X_1,\dots,X_n$, which, in particular, leads to formula giving 
each particular formal weighted double Hurwitz number $h_{g,(m_1,\dots,m_n)}$ in a closed form. Strictly speaking, this Section is not necessary for the rest of the paper, but it sets up the notation and illuminates the logic of computations in the subsequent parts of the paper.

In Section~\ref{sec:CompDH} we derive an explicit closed formula for $D_1\cdots D_n\Hc_{g,n}$. In Section~\ref{sec:MainTh} we prove the main theorem of the present paper, which explicitly represents $\Hc_{g,n}$ for given $g$ and $n$ in a closed form. Section~\ref{sec:exceptions} deals with the slightly exceptional cases of $n=1$ for any $g$ and $(g,n)=(0,2)$. Finally, in Section~\ref{sec:applying} we give examples of the application of our main general formula, deriving explicit closed formulas for $H_{g,n}$ for particular small~$g$ and~$n$.

\subsection{Acknowledgments} S.~S. was supported by the Netherlands Organization for Scientific Research. The research of B.~B. and P.~D.-B. was supported by the Russian Science Foundation (project 20-61-46005).

This project has started when S.~S. was visiting the Faculty of Mathematics at the National Research University Higher School of Economics, and S.~S. would like to thank the Faculty  for warm hospitality and stimulating research atmosphere.

We would like to thank A.~Alexandrov and J.~van~de~Leur for helpful remarks.

\section{Operators on the Fock space} \label{sec:OperFockSpace}

By the (bosonic) \emph{Fock space} we mean the space of infinite power series $\cF=\C[[p_1,p_2,\dots]]$. It has a distinguished element $1$ called \emph{vacuum vector} and denoted sometimes by $\big| 0 \big\rangle$, and a distinguished linear function $\cF\to\C$ called \emph{covacuum vector} that takes a series to its free term (the value at $p=0$) and is denoted by $\big\langle 0 \big|$.

We will consider some operators acting on the Fock space. In particular, we set $J_m=m\,\partial_{p_m}$ if $m>0$, $J_0=0$, and $J_m=p_{-m}$ (the operator of multiplication by $p_{-m}$), if $m<0$. Note that
\begin{equation}\label{eq:Jcomm}
[J_k , J_l] = k\delta_{k+l,0}.
\end{equation}

Introduce also the operator $\mathcal{D}(\hbar)$ acting diagonally in the basis of Schur functions by
\begin{equation} \label{eq:action-of-D}
\mathcal{D}(\hbar)\,s_\lambda = 
e^{\sum\limits_{(i,j)\in\lambda} \psi(\hbar(j-i))}\,s_\lambda.
\end{equation}

With these notations, and using the identity $\sum_\lambda s_\lambda(p)s_\lambda(s)=e^{\sum_{i=1}^\infty s_ip_i/i}$ for Schur polynomials, the definitions of the Orlov-Scherbin partition function and the disconnected $n$-point functions can be rewritten as follows
\begin{align}
Z&=\mathcal{D}(\hbar)e^{\sum_{i=1}^\infty \frac{s_iJ_{-i}}{i\,\hbar}}\big|0\big\rangle,\\
\Hd_n&=\sum_{m_1,\dots,m_n=1}^\infty \frac{X_1^{m_1}\dots X_n^{m_n}}{m_1\dots m_n}
 \VEV{J_{m_1}\dots J_{m_n}\mathcal{D}(\hbar)e^{\sum_{i=1}^\infty \frac{s_iJ_{-i}}{i\,\hbar}}}.\label{eq:Hd}
\end{align}

The introduced standard terminology and notations come from physics. It might look as an unnecessary complication from the first glance; its benefit will be seen later.

A bigger set of operators of our interest is constructed as follows.

\begin{definition}
	The Lie algebra $A_\infty$ is the $\C$-vector space of infinite matrices $(A_{i,j})_{i,j\in \Z+\frac{1}{2}}$ with
	only finitely many non-zero diagonals (that is, $A_{i,j}$ is not equal to zero only for finitely
	many possible values of $i-j$), together with the commutator bracket. The standard basis is formed by the matrix units $\{E_{i,j}\,|\, i,j \in \Z+\frac{1}{2}\}$ such that $(E_{i,j})_{k,l} = \delta_{i,k}\delta_{j,l}$.
\end{definition}

There is a remarkable projective representation of this algebra in the Fock space by means of differential operators. It is denoted by the hat symbol and defined by the following generating function for the action of the matrix units~\cite{Kac,MiwaJimboDate}:
\begin{equation}\label{eq:A-action}
\sum_{k,\ell\in\Z+\frac12}x^{\ell} y^{-k} \hat E_{k,\ell}=
x^{1/2}y^{1/2}\frac{
e^{\sum_{i=1}^\infty (y^{-i}-x^{-i}) \frac{p_i}{i} }
e^{\sum_{i=1}^\infty (x^i-y^i)\partial_{p_i}}-1}
{x-y}.
\end{equation}
The expansion of the exponents on the right hand side enlists all possible monomial differential operators in $p$-variables. The coefficient of any such monomial differential operator, after cancellation, is a polynomial in the half-integer powers of~$x$ and~$y$. The contribution of this operator to $\hat E_{k,\ell}$ is equal to the coefficient of
$x^{\ell} y^{-k}$ in that polynomial.

The term projective representation means that the commutator of matrices from $A_\infty$ corresponds to the commutator of their action on the Fock space up to a scalar operator. More explicitly, we have:
\begin{equation}
\label{E:comm}
[\hat E_{a,b},\hat E_{c,d}] = \delta_{b,c}\hat E_{a,d} - \delta_{a,d}\hat E_{c,b}+\delta_{b,c}\delta_{a,d}(\delta_{b>0}-\delta_{d>0})\mathrm{Id}.
\end{equation}
Equivalently, we have actually a representation of the central extension $A_{\infty}+\C\mathrm{Id}$.

The actual definition of the action of $A_\infty$ in the Fock space goes through fermionic realization of the Fock space and the boson-fermion correspondence, see~\cite{MiwaJimboDate} for the details. But as long as the formula~\eqref{eq:A-action} is established it can be taken as a definition and most part of the underlying formalism can be omitted. The profit of using this representation is that while manipulating with operators it is much easier to make computations directly in the algebra $A_\infty$ rather than in its more complicated action in the Fock space.

However, we will need one more relation that does not follow immediately from~\eqref{eq:A-action}. Namely, any diagonal matrix $\sum_{k\in \Z+\frac12}w_k E_{k,k}\in A_\infty$ acts diagonally in the Schur basis and the corresponding eigenvalue is determined by
\begin{align}\label{eq:Ewdef}
\sum_{k\in \Z+\frac12}w_k \hat E_{k,k} \;s_\lambda=\sum_{i=1}^{\ell(\lambda)}(w_{\lambda_i-i+\frac12}-w_{-i+\frac12})\;s_\lambda=\sum_{(i,j)\in\lambda} v_{j-i}\;s_\lambda,
\end{align}
where
\begin{equation}
v_k=w_{k+\frac12}-w_{k-\frac12},
\end{equation}
see \cite{KazLand} for details. In particular, for the operator $\mathcal{D}(\hbar)$ introduced above we have
\begin{equation}
\mathcal{D}(\hbar)=\exp\left(\sum_{k\in \Z+\frac12}w_k \hat E_{k,k}\right)
\end{equation}
where $w_k$ is determined from relations $w_{k+\frac12}-w_{k-\frac12}=\psi(\hbar\,k)$, $k\in\Z$.

Define
\begin{equation}\label{eq:cEdef}
\cE(u,z):=\sum_{m\in\Z}z^m \sum_{k\in\Z+\frac12}
e^{u(k-\frac{m}{2})} E_{k-m,k}.
\end{equation}
Let
\begin{equation}\label{eq:Sdef}
\mathcal{S}(z)= \dfrac{e^{z/2}-e^{-z/2}}{z}.
\end{equation}
Then, setting $x=z\,e^{u/2}$, $y=z\,e^{-u/2}$ in~\eqref{eq:A-action}, we obtain

\begin{proposition} We have:
\begin{align}\label{eq:cEform}
\cE(u,z)
&=\frac{
	e^{\sum_{i=1}^\infty u\,\cS(u\,i) J_{-i}z^{-i}}
	e^{\sum_{i=1}^\infty u\,\cS(u\,i) J_{i}z^{i}} -1 }
{u\,\cS(u)}.
\end{align}	
\end{proposition}

An independent proof of the equality of coefficients of $z^0$ of both sides can be found in~\cite{SSZ12}.

For example, comparing coefficients of $z^mu^0$ of both sides we find
\begin{equation}
J_m=\sum_{k\in\Z+\frac12}\hat E_{k-m,k}.
\end{equation}
The commutation relation~\eqref{eq:Jcomm} for these operators also implies the following formula:
\begin{proposition}	
	\begin{equation} \label{eq:expJcomm}
	e^{\sum_{i=1}^\infty a_i J_i}
	e^{\sum_{i=1}^\infty b_i J_{-i}}=
	e^{\sum_{i=1}^\infty i \,a_ib_i}
	e^{\sum_{i=1}^\infty b_i J_{-i}}
	e^{\sum_{i=1}^\infty a_i J_i}
	\end{equation}	
	for any collection of constants $a_i,b_i$ such that the corresponding infinite sums make sense.
\end{proposition}
\begin{proof}
This is just a very well-known common special case of the Baker--Campbell--Hausdorff formula, but it is illuminating to see how in this particular case it is just a manifestation of the Taylor formula. Namely, by the Taylor formula, the action of the operator  $e^{\sum_{i=1}^\infty a_i J_i}=e^{\sum_{i=1}^\infty i\,a_i \partial_{p_i}}$ on a series $f(p_1,p_2,\dots)$ results in a shift of the arguments,
\begin{equation}
e^{\sum_{i=1}^\infty a_i J_i}f(p_1,p_2,\dots)
=f(p_1+1\,a_1,p_2+2\,a_2,\dots).
\end{equation}
Therefore, we have
\begin{align}
e^{\sum_{i=1}^\infty a_i J_i}
e^{\sum_{i=1}^\infty b_i J_{-i}}f(p_1,p_2,\dots)&=
e^{\sum_{i=1}^\infty b_i (p_i+i\,a_i)}
f(p_1+1\,a_1,p_2+a\,a_2,\dots)
\\  \notag &
=e^{\sum_{i=1}^\infty i\,a_ib_i}
e^{\sum_{i=1}^\infty b_i p_i}
e^{\sum_{i=1}^\infty i\,a_i \partial_{p_i}}f(p_1,p_2,\dots),
\end{align}
which proves the formulated above commutation relation.
\end{proof}

\section{Preliminary computation of \texorpdfstring{$H_{g,n}$}{Hgn}} \label{sec:preliminary}

In this section we compute $H_{g,n}$ as a series in $X_1,\dots,X_n$. In particular, this leads to a computation of each particular weighted double Hurwitz number $h_{g,(m_1,\dots,m_n)}$ in a closed form.

\subsection{Vacuum expectation expression for \texorpdfstring{$\Hd_n$}{Hn disconnected}}
Let us define
\begin{equation}
\euro_m:=\cD(\hbar)^{-1}J_{m}\cD(\hbar).
\end{equation}
This allows us to rewrite \eqref{eq:Hd} as
\begin{align}\label{eq:Hdisc1}
\Hd_n
&=\sum_{m_1,\dots,m_n=1}^\infty
\frac{X_1^{m_1}\dots X_n^{m_n}}{m_1\dots m_n}
\VEV{\euro_{m_1}\dots \euro_{m_n}e^{\sum_{i=1}^\infty \frac{s_i J_{-i}}{i\hbar} }}.
\end{align}

\begin{proposition}\label{prop:OperJJ}
	The operators $\euro_m(\hbar)$ belong to $A_\infty$ for all $m\in\Z$, namely,
	\begin{equation}\label{eq:DJD}
\euro_m(\hbar) =\sum_{k\in\Z+\frac12}\phi_m(\hbar\,(k-\tfrac m2))\hat E_{k-m,k}.
	\end{equation}
where
\begin{align}\label{eq:phimdef}
\phi_m(y)
&:=\exp\left(\sum_{i=1}^{m}\psi\left(y+\dfrac{2i-m-1}{2} \hbar\right)\right)
,&m>0,\\ \label{eq:phi0def} 
\phi_0(y)&:=1,&\\ 
\phi_m(y)&:=(\phi_{-m}(y))^{-1},&m<0.
\end{align}

More explicitly, we have
	\begin{equation}\label{eq:euroJ}
	\euro_m = \sum_{r=0}^\infty\partial_y^r\phi_m(y)\bigm|_{y=0}[u^rz^m]
	\frac{
		e^{\sum_{i=1}^\infty u\,\hbar\cS(u\,\hbar\,i) J_{-i}z^{-i}}
		e^{\sum_{i=1}^\infty u\,\hbar \cS(u\,\hbar\,i) J_{i}z^{i}}}
	{u\,\hbar\cS(u\,\hbar)}.
	\end{equation}
\end{proposition}

\begin{notation}
Here and below $[x^k]f(x)$ stands for the coefficient in front of $x^k$ in the series expansion of $f(x)$.	
\end{notation}

\begin{proof}[Proof of Proposition~\ref{prop:OperJJ}] For $m=0$ the statement is evident: from \eqref{eq:Ewdef}, the operator $\sum_{k\in\Z+\frac12}\hat E_{k,k}$ annihilates the whole Fock space. Let $m\neq 0$. Recall that $J_m=\sum_{k\in\Z+\frac12}\hat{E}_{k-m,k}$ and
$
\mathcal{D}(\hbar)=\exp(W)
$,
where $W=\sum_{k\in \Z+\frac12}w_k \hat E_{k,k}$ is represented by a diagonal matrix whose diagonal entries $w_k$ are determined from relations 
$w_{k}-w_{k-1}=\psi\left(\hbar\,(k-\frac{1}{2})\right)$, $k\in\Z$
.
Therefore 
using \eqref{E:comm} and the Hadamard's formula $e^XY e^{-X} = e^{\mathrm{ad}_X}(Y)$, where $\mathrm{ad}_X(\cdot) =[X;\cdot]$, we get
\begin{align}
	\euro_m & =e^{-W}\left(\sum_{k\in\Z+\frac12}\hat{E}_{k-m,k}\right) e^{W}
	\\ \notag
   &  =\sum_{k\in\Z+\frac12}e^{w_k-w_{k-m}}\hat{E}_{k-m,k}
   \\ \notag
    & =\sum_{k\in\Z+\frac12}\phi_m(\hbar\,(k-\tfrac m2))\hat{E}_{k-m,k}.
\end{align}

For the proof of~\eqref{eq:euroJ} we compute:
\begin{align}
\euro_m
&\mathop{=}^{\eqref{eq:DJD}}
\sum_{k\in\Z+\frac12}\phi_m(\hbar\,(k-\tfrac m2))\hat{E}_{k-m,k}
\\ \nonumber
&=
\sum_{k\in\Z+\frac12}\sum_{r=0}^\infty\partial_y^r\phi_m(y)\bigm|_{y=0}
\frac{(\hbar\,(k-\tfrac m2))^r}{r!}\hat{E}_{k-m,k}
\\ \nonumber
&\mathop{=}^{\eqref{eq:cEdef}}
\sum_{r=0}^\infty\partial_y^r\phi_m(y)\bigm|_{y=0}[u^rz^m]\cE(u\,\hbar,z)
\\ \nonumber
&\mathop{=}^{\eqref{eq:cEform}}
\sum_{r=0}^\infty\partial_y^r\phi_m(y)\bigm|_{y=0}[u^rz^m]
\frac{
	e^{\sum_{i=1}^\infty u\,\hbar\cS(u\,\hbar\,i) J_{-i}z^{-i}}
	e^{\sum_{i=1}^\infty u\,\hbar \cS(u\,\hbar\,i) J_{i}z^{i}}}
{u\,\hbar\cS(u\,\hbar)}.
\end{align}
In the second line we have simply expanded $\phi_m(\hbar\,(k-\tfrac m2))$ in its Taylor series at zero.
\end{proof}

\subsection{Computation of \texorpdfstring{$\Hd_n$}{Hn disconnected}}

Now we can obtain the following expression for the disconnected $n$-point functions. Let
\begin{equation}
\cS(u)=\frac{e^{u/2}-e^{-u/2}}{u}=\sum_{k=0}^\infty\frac{u^{2k}}{2^{2k}(2k+1)!}.
\end{equation}

\begin{definition}\label{def:U-operators}
Denote by $U^+$ the transformation that takes a Laurent series $f(u,z)$ in $u$ and $z$ to the series in~$X$ given by
\begin{equation}\label{eq:Uplusdef}
(U^+f)(X)=\sum_{m=1}^\infty \frac{X^{m}}{m}\sum_{r=0}^\infty
\partial_y^{r}\phi_{m}(y)\bigm|_{y=0}[z^mu^r]
 \frac{e^{ u\cS(u\hbar\,z\partial_z)y(z) }}{u\hbar\cS(u\hbar)}\;f(u,z).
\end{equation}
This formula describes explicitly the coefficients of $U^+f$ as a power series in~$X$. It makes sense if $f$ is polynomial in~$u$ or if~$f$ is a series in~$\hbar$ whose coefficients are polynomial in~$u$. Remark that $U^+f$ is a regular series in~$X$ even though the series~$f$ might have a pole in~$z$ at the origin: the non-positive powers of~$z$ in the expansion of~$f$ are just ignored.

Denote also by $U^+_k$ a similar transformation applied to $u_k$ and $z_k$ instead of $u$ and $z$ (the output of $U^+_k$ is a power series in $X_k$).
\end{definition}

In all relations of this section 
the functions on the right hand sides are understood as power asymptotic expansion in the sector $|z_1|\ll\dots\ll|z_n|\ll1$.

\begin{proposition}\label{prop:HdU}
	We have
	\begin{equation}\label{eq:Hdprod}
	\Hd_n=U^+_n\dots U^+_1
	\prod_{1\le k<\ell\le n}
	e^{\hbar^2u_ku_\ell\cS(u_k\hbar\,z_k\partial_{z_k})\cS(u_\ell\hbar\,z_\ell \partial_{z_\ell})
		\frac{z_k z_\ell}{(z_k-z_\ell)^2}},
	\end{equation}
where the expression in the product on the right hand side is understood as its power asymptotic expansion in the sector $|z_1|\ll\dots\ll|z_n|\ll1$.
\end{proposition}

\begin{proof}
Let us substitute expressions \eqref{eq:euroJ} for $\euro$-operators into \eqref{eq:Hdisc1}. We get
\begin{align}
\Hd_n&=\sum_{m_1,\dots,m_n=1}^\infty
\sum_{r_1,\dots,r_n=0}^\infty
\left(\prod_{k=1}^n\partial_y^{r_k}\phi_{m_k}(y)\bigm|_{y=0}\frac{X_k^{m_k}}{m_k}\right)\times
\\  \nonumber
&\qquad\left[\textstyle \prod_{i=1}^n z_i^{m_i}u_i^{r_i}\right]
\VEV{\prod_{k=1}^n\dfrac{
		e^{\sum_{i=1}^\infty u_k\,\hbar\cS(u_k\,\hbar\,i) J_{-i}z_k^{-i}}
		e^{\sum_{i=1}^\infty u_k\,\hbar \cS(u_k\,\hbar\,i) J_{i}z_k^{i}}}
	{u_k\,\hbar\cS(u_k\,\hbar)}\;\;
e^{\sum_{i=1}^\infty \frac{s_i J_{-i}}{i\hbar} }
}.
\end{align}
Then we apply commutation relations \eqref{eq:expJcomm} for the exponents of $J$-operators moving the $J_{>0}$-factors to the right and the $J_{<0}$-factors to the left. Since $J_{>0}$ is killed by the vacuum vector and $J_{<0}$ is killed by the covacuum, we get
\begin{align}
&\VEV{\prod_{k=1}^n\dfrac{
		e^{\sum_{i=1}^\infty u_k\,\hbar\cS(u_k\,\hbar\,i) J_{-i}z_k^{-i}}
		e^{\sum_{i=1}^\infty u_k\,\hbar \cS(u_k\,\hbar\,i) J_{i}z_k^{i}}}
	{u_k\,\hbar\cS(u_k\,\hbar)}\;\;
	e^{\sum_{i=1}^\infty \frac{s_i J_{-i}}{i\hbar} }}
\\ \nonumber
&= \prod_{k=1}^n \dfrac{\exp\left(\sum_{i=1}^\infty u_k\cS(u_k\hbar\,i)s_iz_k^i \right)}
{u_k\hbar\cS(u_k\hbar)}\;\;
\prod_{1\le k<\ell\le n}
\exp\left(\sum_{i=1}^\infty u_k\hbar\cS(u_k\hbar\,i)\,u_\ell\hbar\cS(u_\ell\hbar\,i)
	\,i\left(\dfrac{z_k}{z_\ell}\right)^i \right).
\end{align}

Recall that
\begin{align}
&\sum_{i=1}^\infty s_i z_k^i=y(z_k)=:y_k.
\end{align}
Also note that
\begin{align}
 \label{eq:bergexp}
&\sum_{i=1}^\infty i\left(\dfrac{z_k}{z_\ell}\right)^i
=z_k\partial_{z_k}\frac{z_\ell}{z_k-z_\ell}=\frac{z_kz_\ell}{(z_k-z_\ell)^2},
\end{align}
if we assume that $z_k \ll z_\ell$. 

Noting all that we finally obtain
\begin{align}
& \Hd_n=\sum_{m_1,\dots,m_n=1}^\infty
\sum_{r_1,\dots,r_n=0}^\infty
\left(\prod_{k=1}^n\partial_y^{r_k}\phi_{m_k}(y)\bigm|_{y=0}\frac{X_k^{m_k}}{m_k}\right) \times\\ \nonumber
&[z_1^{m_1}\dots z_n^{m_n}u_1^{u_1}\dots u_n^{r_n}]
\prod_{k=1}^n \frac{e^{ u_k\cS(u_k\hbar\,z_k\partial_{z_k})y_k }}
{u_k\hbar\cS(u_k\hbar)}\;\;
\prod_{1\le k<\ell\le n}
e^{\hbar^2u_ku_\ell\cS(u_k\hbar\,z_k\partial_{z_k})\cS(u_\ell\hbar\,z_\ell\partial_{z_\ell})
	\frac{z_k z_\ell}{(z_k-z_\ell)^2}},
\end{align}
where the expression in the second line after $[z_1^{m_1}\dots z_n^{m_n}u_1^{u_1}\dots u_n^{r_n}]$ is understood as its power asymptotic expansion in the sector $|z_1|\ll\dots\ll|z_n|\ll1$. This formula is equivalent to that of Proposition.
\end{proof}

\begin{remark}
	Note that the argument of $U^+$-operators in \eqref{eq:Hdprod} involves both positive and negative powers of the variables~$z_k$ but the left hand side is determined by those monomials of the right hand side that contain positive powers of all variables only.
\end{remark}

\subsection{From disconnected to connected \texorpdfstring{$n$}{n}-point functions}

With notations of the previous section, we have:
\begin{proposition}\label{prop:Hnconn}
\begin{equation}\label{eq:Hnconn}
H_n
=U^+_n\dots U^+_1
\sum_{\gamma \in \Gamma_n}\prod_{\{v_k,v_\ell\}\in E_\gamma}
\left(e^{\hbar^2u_ku_\ell\cS(u_k\hbar\,z_k\partial_{z_k})\cS(u_\ell\hbar\,z_\ell\partial_{z_\ell})
	\frac{z_k z_\ell}{(z_k-z_\ell)^2}}-1\right),
\end{equation}	
where $\Gamma_n$ is the set of all connected simple (i.e. without multiple edges and loops) graphs over $n$ vertices $v_1,\ldots,v_n$, and $E_\gamma$ is the set of edges of $\gamma \in \Gamma_n$.
\end{proposition}

\begin{proof}
Let us denote
\begin{equation}
w_{k,\ell}=
e^{\hbar^2u_ku_\ell\cS(u_k\hbar\,z_k\partial_{z_k})\cS(u_\ell\hbar\,z_\ell\partial_{z_\ell})
	\frac{z_k z_\ell}{(z_k-z_\ell)^2}}-1
\end{equation}
and consider the product
\begin{equation}
\prod_{1\le k<\ell\le n}
e^{\hbar^2u_ku_\ell\cS(u_k\hbar\,z_k\partial_{z_k})\cS(u_\ell\hbar\,z_\ell\partial_{z_\ell})
	\frac{z_k z_\ell}{(z_k-z_\ell)^2}}
=\prod_{1\le k<\ell\le n}
(1+w_{k,\ell}).
\end{equation}
Expanding the brackets we obtain $2^{\binom{n}{2}}$ summands. These summands are labeled by simple graphs on $n$ numbered vertices: the vertices~$k$ and~$\ell$ are connected or not connected by an edge if the factor corresponding to the pair of indices~$k$ and~$\ell$ is equal to~$w_{k,\ell}$ or~$1$, respectively.


Then, Equation~\eqref{eq:Hdprod} for the disconnected $n$-point functions attains the following form
\begin{equation}
\Hd_n=U^+_n\dots U^+_1
\sum_{\gamma}\prod_{\{v_k,v_\ell\}\in E_\gamma} w_{k,\ell},
\end{equation}
where the summation carries over the set of \emph{all} simple graphs~$\gamma$ on $n$ labeled vertices. The inclusion-exclusion procedure applied to this sum over all simple graphs singles out exactly the terms corresponding to the connected ones.
\end{proof}

It is sometimes convenient to rearrange the insertion of $\hbar$ in Equation~\eqref{eq:Hnconn} in the following way
\begin{equation}
\hbar^{2-n}H_n=
(\hbar U^+_n)\dots(\hbar U^+_1)
\sum_{\gamma \in \Gamma_n}\hbar^{2(|E_\gamma|-n+1)}\prod_{\{v_k,v_\ell\}\in E_\gamma}
\frac{w_{k,\ell}}{\hbar^2},
\end{equation}	
Since any connected graph on $n$ vertices has at least $n-1$ edges, the right hand side involves only non-negative even powers of the variable $\hbar$. Indeed, it is easy to see from definition that the series $\frac{w_{k,\ell}}{\hbar^2}$ and the coefficients of the transformation $\hbar U^+$ involve nonnegative even powers of $\hbar$ only. This justifies in a formal way the mentioned  \emph{genus decomposition}
\begin{equation}\label{eq:Hgndef}
\hbar^{2-n}\Hc_n=\sum_{g=0}^\infty\hbar^{2g}\Hc_{g,n}
\quad\text{or}\quad
\Hc_n=\sum_{g=0}^\infty\hbar^{2g-2+n}\Hc_{g,n}
\end{equation}
where $\Hc_{g,n}$ is independent of $\hbar$.

Finally, note that the operators $U^+_i$ describe explicitly the Taylor coefficients of the resulting series. Therefore, we can regard~\eqref{eq:Hnconn} as an explicit expression for the corresponding Hurwitz numbers:
\begin{align}\label{eq:hexpr}
& m_1\dots m_n\;h_{g,(m_1,\dots,m_n)}=
\\ \notag & \qquad
[\hbar^{2g-2+n}]
\sum_{r_1,\dots,r_n=0}^\infty
\left(\prod_{k=1}^n\partial_y^{r_k}\phi_{m_k}(y)\bigm|_{y=0}\right) [z_1^{m_1}\dots z_n^{m_n}u_1^{r_1}\dots u_n^{r_n}]
\\ \notag & \qquad
\prod_{k=1}^n \frac{e^{ u_k\cS(u_k\hbar\,z_k\partial_{z_k})y_k }}
{u_k\hbar\cS(u_k\hbar)}\;\;
\sum_{\gamma\in\Gamma_n}  \prod_{\{v_k,v_\ell\}\in E_\gamma}
\left(e^{\hbar^2u_ku_\ell\cS(u_k\hbar\,z_k\partial_{z_k})\cS(u_\ell\hbar\,z_\ell\partial_{z_\ell})
	\frac{z_k z_\ell}{(z_k-z_\ell)^2}}-1\right).
\end{align}

\begin{remark}
Formulas \eqref{eq:Hnconn} and \eqref{eq:hexpr} provide closed expressions for the connected $n$-point functions and connected formal weighted double Hurwitz numbers as sums over graphs, respectively. However, note that our main aim, as explained in the introduction, is to express the connected $n$-point functions as finite polynomials in certain formal functions on the spectral curve, and formula \eqref{eq:Hnconn} does not achieve that. Indeed, note that in the definition \eqref{eq:Uplusdef} of the operator $U^+$ we have an infinite sum over $m$. It turns out that, roughly speaking, it is possible to 
take these $m$-sums to arrive at finite expressions, and this is what is done in the two following sections. However, the precise path to arriving at these finite expressions, while being inspired by the contents of the present section, does not explicitly rely on Proposition \ref{prop:Hnconn} and is, strictly speaking, independent of this section. We do use the notation introduced in the present section in what follows; notably, $\phi_m$'s will play an important role.
\end{remark}




\section{Computation of \texorpdfstring{$D_1\dots D_n H_n$}{D1...DnHn}} \label{sec:CompDH}

Set
\begin{equation}
D_i=X_i\partial_{X_i}.
\end{equation}
Denote, for shortness,
\begin{equation}
DH^\bullet_n=\left(\prod_{i=1}^n D_i\right)\Hd_n\qquad \text{and} \qquad
DH_n=\left(\prod_{i=1}^n D_i\right)H_n
 =\sum_{g=0}^\infty \hbar^{2g-2+n}DH_{g,n}.
\end{equation}
In this section we compute these functions in a closed form. Remark that the operator $D_1\dots D_n$ multiplies a monomial $X_1^{m_1}\dots X_n^{m_n}$ by the factor $m_1\dots m_n$. Since both $H_n$ and $\Hd_n$ only involve monomials with $m_i>0$, the series $DH_n$ and $D\Hd_n$ determine uniquely the original series $H_n$ and $\Hd_n$, respectively.

\subsection{Completed \texorpdfstring{$n$}{n}-point function}
We have from~\eqref{eq:Hd}
\begin{equation}\label{eq:DHd}
D\Hd_n=\sum_{m_1,\dots,m_n=1}^\infty X_1^{m_1}\dots X_n^{m_n}
 \VEV{J_{m_1}\dots J_{m_n}\mathcal{D}(\hbar)e^{\sum_{i=1}^\infty \frac{s_iJ_{-i}}{i\,\hbar}}}.
\end{equation}

Define \emph{completed} version of this function by
\begin{equation}\label{eq:cDHd}
\widehat{DH}^\bullet_n=\sum_{m_1,\dots,m_n=-\infty}^\infty X_1^{m_1}\dots X_n^{m_n}
 \VEV{J_{m_1}\dots J_{m_n}\mathcal{D}(\hbar)e^{\sum_{i=1}^\infty \frac{s_iJ_{-i}}{i\,\hbar}}}
 \end{equation}
and the corresponding completed connected functions $\widehat{D H}_n=\sum_{g=0}^\infty \hbar^{2g-2+n}\widehat{D H}_{g,n}$ though inclusion-exclusion relations similar to~\eqref{eq:inclexclH}.

These are infinite power series that involve both positive and negative powers of the variables $X_i$. The advantage of using completed versions of $n$-point functions is that they are better adapted to convolving in a closed form, as we shall see below.

\begin{proposition}\label{prop:H_corr}
We have
\begin{equation}\label{eq:H_corr}
\widehat{DH}_n=DH_n+\delta_{2,n}\frac{X_1X_2}{(X_1-X_2)^2},
\end{equation}
where the last summand is considered as its power expansion over $X_1/X_2$:
\begin{equation}
\frac{X_1X_2}{(X_1-X_2)^2}=\sum_{m=1}^\infty m\left(\tfrac{X_1}{X_2}\right)^m.
\end{equation}
In other words, for $(g,n)\ne(0,2)$ we have
\begin{equation}
\widehat{D H}_{g,n}=DH_{g,n}
\end{equation}
and
\begin{equation}\label{eq:DH02}
\widehat{D H}_{0,2}=DH_{0,2}+\frac{X_1X_2}{(X_1-X_2)^2}.
\end{equation}

As a corollary, for $(g,n)\ne(0,2)$ the series $\widehat{D H}_{g,n}$ involves positive powers of the variables $X_i$ only.
\end{proposition}

\begin{proof}
Denote
\begin{equation}
\nabla^+_i=\sum_{m=1}^\infty X_i^m J_m=\sum_{m=1}^\infty m X_i^m \partial_{p_m},\quad
\nabla^-_i=\sum_{m=1}^\infty X_i^{-m} J_{-m}=\sum_{m=1}^\infty X_i^{-m} p_m.
\end{equation}

Using these operators, we can rewrite~\eqref{eq:DHd} and~\eqref{eq:cDHd} as
\begin{align}
D\Hd_n&=
\nabla_1^+\dots \nabla_n^+Z\Big|_{p=0},\\
\widehat{DH}^\bullet_n&=
(\nabla_1^++\nabla_1^-)\dots (\nabla_n^++\nabla_n^-)Z\Big|_{p=0}.
\end{align}
Let us expand brackets in the last equation. By the Leibniz rule, the partial derivatives entering $\nabla_i^+$ are applied to either the linear functions entering $\nabla_j^-$ for some $j>i$ or to~$Z$. Therefore, we obtain
\begin{equation}\label{eq:Hdcorr}
\widehat{DH}^\bullet_n=
\sum_{\{1,\dots,n\}=\sqcup_k\{i_k,j_k\}\bigsqcup K}\left(\prod_k\frac{X_{i_k}X_{j_k}}{(X_{i_k}-X_{j_k})^2}\right)D\Hd_{|K|}(X_K),
\end{equation}
where the factor $\frac{X_iX_j}{(X_i-X_j)^2}$ for $i<j$ is considered as a power expansion
\begin{equation}
\nabla^+_i\sum_{m=1}^\infty X_j^{-m} p_m
=\sum_{m=1}^\infty m\left(\tfrac{X_i}{X_j}\right)^m
=\frac{X_iX_j}{(X_i-X_j)^2}.
\end{equation}
By inclusion-exclusion relations, Equation~\eqref{eq:Hdcorr} is equivalent to relations of Proposition. In order to see this, we observe that if we \emph{define} connected functions $\widehat{DH}_n$ by~\eqref{eq:H_corr}
then the corresponding disconnected functions are given exactly by~\eqref{eq:Hdcorr}.
\end{proof}

\subsection{Computation of completed \texorpdfstring{$n$}{n}-point functions}

The computation of $\Hd_n$ and $H_n$ of the previous section can be extended to the computation of the completed $n$-point functions $\widehat{DH}^\bullet_n$ and $\widehat{DH}_n$. We represent the corresponding statements but skip the proofs since they are the same, we just extend all summations over $m_i\ge1$ to the summations over $m_i\in Z$.

Define transformation $U$ taking a Laurent series $f(u,z)$ in $u$ and $z$ to the Laurent series
\begin{equation}
(U f)(X)=\sum_{m=-\infty}^\infty X^{m}\sum_{r=0}^\infty
\partial_y^{r}\phi_{m}(y)\bigm|_{y=0}[z^mu^r]
 \frac{e^{ u\cS(u\hbar\,z\partial_z)y(z) }}{u\cS(u\hbar)}\;f(u,z).
\end{equation}
It differs from the transformation $U^+$ of Definition~\ref{def:U-operators} by an extra factor $m$ of the summands and by the summation range of the integer index~$m$. Thus $U f$ is a Laurent series and might involve negative powers of~$X$. We denote also by $U_k$ a similar transformation applied to $u_k$ and $z_k$ instead of $u$ and $z$ (the output of $U_k$ is a Laurent series in $X_k$).

Then, similarly to the computation of $\Hd_n$ we obtain
\begin{equation}
\begin{aligned}
\widehat{DH}^\bullet_n
&=\sum_{m_1,\dots,m_n=-\infty}^\infty
X_1^{m_1}\dots X_n^{m_n}
\VEV{J_{m_1}\dots J_{m_n}\mathcal{D}(\hbar)e^{\sum_{i=1}^\infty \frac{s_i J_{-i}}{i\hbar} }}\\
&=\sum_{m_1,\dots,m_n=-\infty}^\infty
X_1^{m_1}\dots X_n^{m_n}
\VEV{\euro_{m_1}\dots \euro_{m_n}e^{\sum_{i=1}^\infty \frac{s_i J_{-i}}{i\hbar} }}\\
&=U_n\dots U_1
	\prod_{1\le k<\ell\le n}
	e^{\hbar^2u_ku_\ell\cS(u_k\hbar\,z_k\partial_{z_k})\cS(u_\ell\hbar\,z_\ell \partial_{z_\ell})
		\frac{z_k z_\ell}{(z_k-z_\ell)^2}}.
\end{aligned}
\end{equation}
where the expression in the product on the right hand side is understood as its power asymptotic expansion in the sector $|z_1|\ll\dots\ll|z_n|\ll1$.

Next, the analogue of the computation of $H_n$ of the previous section is the following equation
\begin{equation}\label{eq:hatDH}
\widehat{DH}_n
=U_n\dots U_1
\sum_{\gamma \in \Gamma_n}\prod_{\{v_k,v_\ell\}\in E_\gamma}
\left(e^{\hbar^2u_ku_\ell\cS(u_k\hbar\,z_k\partial_{z_k})\cS(u_\ell\hbar\,z_\ell \partial_{z_\ell})
	\frac{z_k z_\ell}{(z_k-z_\ell)^2}}-1\right),
\end{equation}	
where $\Gamma_n$ is the set of all connected simple graphs over $n$ vertices $v_1,\ldots,v_n$, and $E_\gamma$ is the set of edges of $\gamma \in \Gamma_n$.

Since, by Proposition~\ref{prop:H_corr},  $\widehat{DH}_n$ differs from $DH_n$ by a small correction for $n=2$, we conclude:

\begin{corollary}\label{cor:preDHgn}
For $(g,n)\ne(0,2)$ we have
\begin{equation}\label{eq:preDHgn}
DH_{g,n}=[\hbar^{2g-2+n}]\left(
U_n\dots U_1
\sum_{\gamma \in \Gamma_n}\prod_{\{v_k,v_\ell\}\in E_\gamma}
\left(e^{\hbar^2u_ku_\ell\cS(u_k\hbar\,z_k\partial_{z_k})\cS(u_\ell\hbar\,z_\ell \partial_{z_\ell})
	\frac{z_k z_\ell}{(z_k-z_\ell)^2}}-1\right)
\right).
\end{equation}
In particular, all the terms on the right hand side containing non-positive powers of the variables~$X_i$ cancel out.
\end{corollary}

\subsection{Principal identity}\label{sec:princid}

Recall that the transformation~$U$ entering formulas of the previous section acts on a Laurent series $f(u,z)$ in~$z$ and~$u$~by
\begin{equation}
(U f)(X)=\sum_{m=-\infty}^\infty X^{m}\sum_{r=0}^\infty
\partial_y^{r}\phi_{m}(y)\bigm|_{y=0}[z^mu^r]
 \frac{e^{ u\cS(u\hbar\,z\partial_z)y(z) }}{u\hbar\cS(u\hbar)}\;f(u,z).
\end{equation}
The result of this transformation is a function in~$X$. Up to this point we regarded~$X$ and~$z$ as independent variables. From now on we assume that they are related by the change $X=X(z)$ where
\begin{equation}\label{eq:Xdef}
X(z)=z\,e^{-\psi(y(z))}.
\end{equation}
Through this change we have
\begin{align}
D&:=X\frac{\partial}{\partial X}=\frac{1}{Q}z\frac{\partial}{\partial z},
\end{align}
where
\begin{align}
Q&:=\frac{z}{X\frac{dz}{dX}}=\frac{z}{X}\frac{dX}{dz}=1-D\psi(y)=1-z\psi'(y)y'(z).
\end{align}

Thus we have
\begin{equation}\label{eq:zdzQD}
z\dfrac{\partial}{\partial z} = QD.
\end{equation}

Having this change in mind we treat the result of transformation $U$ as a function (a Laurent series) in~$z$. We claim that \emph{$U$ acts on the coefficients of positive powers of~$u$ as a differential operator}. More explicitly, define
\begin{align}
L_0(v,y,\hbar)&:=e^{v\left(\frac{\cS(v\hbar\partial_y)}{\cS(\hbar\partial_y)}-1\right)\psi(y)},\\ \label{eq:Lrdef}
L_r(v,y,\hbar)&:=e^{-v\psi(y)}\partial_y^re^{v\psi(y)}L_0(v,y,\hbar)
=\left(\partial_y+v\psi'(y)\right)^rL_0(v,y,\hbar).
\end{align}
The function $L_r(v,y,\hbar)$ is a series in $\hbar^2$ whose coefficients are polynomials in $v$ and the higher order derivatives of $\psi(y)$.

The following \emph{principal identity} plays a central role in the proof of the main theorems~\ref{prop:mainprop} and~\ref{th:MainTh} below.
\begin{proposition}\label{prop:principal}
Let $H(u,z)$ be arbitrary Laurent series in~$z$ whose coefficients are either polynomials in~$u$ or infinite series in~$\hbar$ such that the coefficient of any power of~$\hbar$ is a polynomial in~$u$. Then the following identity holds true:
\begin{align}\label{eq:princid}
& \sum_{m=-\infty}^\infty \sum_{r=0}^\infty \partial_y^{r}\phi_{m}(y)\bigm|_{y=0}
X^m[z^m u^r]e^{u\,y(z)}H(u,z)\\ \notag
& \qquad =\sum_{j,r=0}^\infty D^j\left(\frac{[v^j]L_r(v,y(z),\hbar)}{Q}[u^r] H(u,z)\right)\,,
\end{align}
where $X=X(z)$ on the left hand side is given by~\eqref{eq:Xdef}.
\end{proposition}

Applying this identity to a function of the form $H(u,z)=\frac{e^{u(\cS(u\,\hbar\,Q\,D)-1)y(z)}}{u\,\cS(u\,\hbar)}f(u,z)$ we conclude:

\begin{corollary}\label{cor:principal}
Assume that $f(u,z)$ is a Laurent series in $z$ whose coefficients are polynomials in~$u$ of bounded degree and with zero free term. Then the action of the transformation~$U$ on~$f$ is given by
\begin{equation}\label{eq:principal}
(U f)(z)=\sum_{j,r=0}^\infty D^j\left(\frac{[v^j]L_r(v,y(z),\hbar)}{Q}[u^r] \frac{e^{u(\cS(u\,\hbar\,Q\,D)-1)y(z)}}{u\,\cS(u\,\hbar)}f(u,z)\right).
\end{equation}
\end{corollary}

\subsection{Proof of the principal identity}
The proof of the principal identity is split into several lemmata.


\begin{lemma}\label{lem:L1}	
Let $\Phi(y)$ and $H(u)$ be arbitrary two regular series. Then
\begin{equation}\label{eq:L1}
\sum_{r=0}^\infty \partial_y^r\Phi(y)\bigm|_{y=0} [u^r]e^{u y}H(u)=
\sum_{r=0}^\infty \partial_y^r\Phi(y) \;[u^r]  H(u).
\end{equation}
\end{lemma}

\begin{proof} We have:
	\begin{align}
	\sum_{r=0}^\infty \partial_y^r\Phi(y)\bigm|_{y=0} [u^r]e^{u y}H(u) & =
	\sum_{r,k=0}^\infty \partial_y^{r+k}\Phi(y)\bigm|_{y=0}
	\Bigl([u^{k}]e^{u y}\Bigr)\Bigl([u^r]H(u)\Bigr)\\ \notag
	& =
	\sum_{r,k=0}^\infty \partial_y^{r+k}\Phi(y)\bigm|_{y=0}\frac{y^k}{k!}\;
	[u^r] H(u) \\ \notag
	& =
	\sum_{r=0}^\infty \partial_y^r\Phi(y)\;  [u^r]H(u).
	\end{align}
\end{proof}



\begin{lemma}[see \cite{Kaz2020}]	
We have:
\begin{align}
\phi_m(y,\hbar)&=e^{m\psi(y)}L_0(m,y,\hbar),\\ \label{eq:dphi_eL}
\partial_y^r\phi_m(y,\hbar)
&=e^{m\psi(y)}L_r(m,y,\hbar).
\end{align}
\end{lemma}

We also need a certain form of what is known as the Lagrange-B\"urmann formula:
\begin{lemma}For any Laurent series $H$ in $z$ and for any $m\in\Z$ we have	
\begin{equation}\label{eq:LagBur}
[z^m] e^{m\psi(y)} H = [X^m] \frac{1}{Q}H,
\end{equation}
and, therefore,
\begin{equation}\label{eq:LagBur1}
\sum_{m=-\infty}^\infty X^m [z^m] e^{m\psi(y)} H = \frac{1}{Q}H,
\end{equation}
where $y=y(z)$ and the function on the right hand side is regarded as a Laurent series in~$X$ though the change inverse to~\eqref{eq:Xdef}.
\end{lemma}

\begin{proof} We have:
	\begin{equation}
	[z^m] e^{m\psi(y)} H =
	\res_{z=0}  \frac{e^{m\psi(y)} H}{z^{m+1}}dz=\\
	\res_{z=0}  \frac{H}{z\,X^{m}}dz=
	\res_{z=0}  \frac{H}{Q\,X^{m+1}}dX= [X^m] \frac{1}{Q}H.
	\end{equation}
\end{proof}

Now we are ready to prove the principal identity.

\begin{proof}[Proof of proposition \ref{prop:principal}] We have:
\begin{align}
\sum_{m=-\infty}^\infty X^{m}&\sum_{r=0}^\infty
\partial_y^{r}\phi_{m}(y)\bigm|_{y=0}[z^mu^r]
 e^{u\,y(z)}H(u,z)
\\ \notag
\mathop{=}^{\eqref{eq:L1}}&
\sum_{m=-\infty}^\infty X^m[z^m]\sum_{r=0}^\infty [u^r]\partial_y^r\phi_m(y,\hbar)\bigm|_{y=y(z)}H(u,z)
\\ \notag
\mathop{=}^{\eqref{eq:dphi_eL}}&\sum_{m=-\infty}^\infty X^m[z^m]\sum_{r=0}^\infty [u^r]
e^{m \psi(y(z))}L_r(m,y(z),\hbar)\;H(u,z)
\\ \notag
=&\sum_{j=0}^\infty D^j\sum_{m=-\infty}^\infty X^m[z^m]\sum_{r=0}^\infty [u^r]
e^{m \psi(y(z))}[v^j]L_r(v,y(z),\hbar)\;H(u,z)
\\ \notag
\mathop{=}^{\eqref{eq:LagBur1}}&\sum_{j=0}^\infty D^j\sum_{r=0}^\infty
\frac{[v^j]L_r(v,y(z),\hbar)}{Q}[u^r]\;
H(u,z)
\end{align}
(here we consider $X$ and $z$ as independent variables).
\end{proof}

\subsection{A closed formula for \texorpdfstring{$D_1\dots D_n H_n$}{D1...DnHn}}

The principal identity together with Corollary~\ref{cor:preDHgn} lead to the first our theorem, which is just one step away from the main result formulated in the next section:

\begin{theorem}\label{prop:mainprop}	For $n\geq 2$, $(g,n)\neq(0,2)$ we have
\begin{equation}\label{eq:mainprop}
D_1\dots D_nH_{g,n}=[\hbar^{2g-2+n}]
U_n\dots U_1
\sum_{\gamma \in \Gamma_n}\prod_{\{v_k,v_\ell\}\in E_\gamma} w_{k,\ell},
\end{equation}
where
\begin{equation}\label{eq:wkl}
w_{k,\ell}=e^{\hbar^2u_ku_\ell\cS(u_k\hbar\,Q_k D_k)\cS(u_\ell\hbar\,Q_\ell D_\ell)
	\frac{z_k z_\ell}{(z_k-z_\ell)^2}}-1
\end{equation}
and $U_{i}$ is the operator of Proposition~\ref{prop:principal} acting on a function~$f$ in~$u_i$ and $z_i$ by
\begin{equation}\label{Uioperator}
U_{i} f=\sum_{j,r=0}^\infty D_i^j\left(\frac{[v^j]L_r(v,y(z_i),\hbar)}{Q_i}
[u_i^r] \frac{e^{u_i(\cS(u_i\,\hbar\,Q_i\,D_i)-1)y(z_i)}}{u_i\hbar\,\cS(u_i\,\hbar)}f(u_i,z_i)\right).
\end{equation}
As before, the sum is over all connected simple graphs 
on $n$ labeled vertices.

For fixed $g$ and $n$, after taking the coefficient $[\hbar^{2g-2+n}]$, all sums in this formula for $\left(\prod_{i=1}^n D_i\right) \Hc_{g,n}$ become finite, and it becomes a rational expression in $z_1,\ldots,z_n$ and the derivatives of the functions $y_i=y(z_i)$ and $\psi(y_i)$.
\end{theorem}

The coefficient of any power of~$\hbar$ in $w_{i,j}$ is a polynomial in $u_i$ and $u_j$ vanishing at $u_i=0$ and at $u_j=0$ so that Corollary~\ref{cor:principal} can be applied.
The restriction $n>1$ is imposed because in the case $n=1$ the operator $U_1$ is applied to the constant function~$1$, which is not divisible by~$u_1$, so the conclusion of Corollary~\ref{cor:principal} does not hold. The requirement $(g,n)\ne(0,2)$ is a consequence of Corollary~\ref{prop:H_corr}. The cases $n=1$ and $(g,n)=(0,2)$ are treated in Section~\ref{sec:exceptions} separately.

The (very important) finiteness statement is evident from the way $\hbar$ and $u_i$ enter the expression.

A nice property of the equality of Theorem~\ref{prop:mainprop} (that does not hold for earlier equalities of Proposition~\ref{prop:Hnconn} and Corollary~\ref{cor:preDHgn})
is that it can be applied without expanding involved functions in Laurent series and is valid in the ring $\mathbb{C}[[z_1,\dots,z_n]][(z_i-z_j)^{-1};i,j=1,\ldots,n\}]$ of functions with finite order poles on the diagonals $z_i=z_j$.

\begin{remark} \label{rem:DHdeformed}
	Note that the statement of Theorem~\ref{prop:mainprop} still holds if one allows $\psi(z)$ and $y(z)$ to also be formal series in $\hbar^2$. More precisely, an analogous statement can be proved in a very similar way if one puts
	\begin{align}
	\psi(\hbar^2,y)&:=\sum_{k=1}^\infty\sum_{m=0}^\infty c_{k,m} y^k\hbar^{2m},\\
	y(\hbar^2,z)&:=\sum_{k=1}^\infty\sum_{m=0}^\infty s_{k,m} z^k\hbar^{2m},
	\end{align}
	while still keeping the formula for $X(z)$ free of $\hbar$, i.e. using
	\begin{equation}
	X(z)=\left.z\,e^{-\psi(y(z))}\right|_{\hbar=0}
	\end{equation}
	in place of \eqref{eq:X(z)}, see~\cite[Section 2]{BDKS21toprec}.
\end{remark}

Theorem~\ref{prop:mainprop} has an important corollary:
\begin{corollary}\label{cor:diagonal}
	All diagonal poles (i.e. poles at $z_i=z_j$ for $i\neq j$) on the right hand side of \eqref{eq:mainprop} cancel out.
\end{corollary}
\begin{proof}
	Indeed, the left hand side manifestly does not have diagonal poles, thus they must cancel out on the right hand side.
\end{proof}

At the end of this section we state several reformulations of Theorem~\ref{prop:mainprop}. First, substituting the definitions of $w_{k,\ell}$ and $U_k$ to~\eqref{eq:mainprop}, we get, explicitly,
\begin{align}\label{eq:Hlongexpr0}
& D_1\dots D_nH_{g,n}=[\hbar^{2g-2+n}]
\sum_{j_1,\dots,j_n,r_1,\dots,r_n=0}^\infty \left(\textstyle\prod_{i=1}^nD_i^{j_i}\right)\Biggl(
\prod_{i=1}^n
\tfrac{[v^{j_i}]L_{r_i}(v,y(z_i),\hbar)}{Q_i}[{\textstyle\prod_{i=1}^n u_i^{r_i}}]
\Biggr.\\\Biggl. \notag
&
\prod_{i=1}^n
 \frac{e^{u_i(\cS(u_i\,\hbar\,Q_i\,D_i)-1)y(z_i)}}{u_i\hbar\,\cS(u_i\,\hbar)}
\sum_{\gamma \in \Gamma_n}\prod_{\{v_k,v_\ell\}\in E_\gamma}
\left(e^{\hbar^2u_ku_\ell\cS(u_k\hbar\,Q_k D_k)\cS(u_\ell\hbar\,Q_\ell D_\ell)
	\frac{z_k z_\ell}{(z_k-z_\ell)^2}}-1\right)\Biggr).
\end{align}
Next, expanding the exponent in a series we can represent the last formula in the following even more explicit form:
\begin{align}\label{eq:Hlongexpr}
& D_1\dots D_nH_{g,n}=[\hbar^{2g-2+n}]
\sum_{j_1,\dots,j_n,r_1,\dots,r_n=0}^\infty \left(\textstyle\prod_{i=1}^nD_i^{j_i}\right)\Biggl(
\prod_{i=1}^n
\frac{[v^{j_i}]L_{r_i}(v,y(z_i),\hbar)}{Q_i}\, [{\textstyle\prod_{i=1}^n u_i^{r_i}}]
\\ \notag
& \prod_{i=1}^n
 \frac{e^{u_i(\cS(u_i\,\hbar\,Q_i\,D_i)-1)y(z_i)}}{u_i\hbar\,\cS(u_i\,\hbar)}
\sum_{\gamma\in\widetilde{\Gamma}_n} \frac{\hbar^{2|E_\gamma|}}{|\mathsf{Aut}_\gamma|} \prod_{\{v_k,v_\ell\}\in E_\gamma} u_ku_\ell \cS(u_k\hbar\,z_k\partial_{z_k})\cS(u_\ell\hbar\,z_\ell\partial_{z_\ell})\frac{ z_kz_\ell}{(z_k-z_\ell)^2}\Biggr).
\end{align}
Here $\widetilde{\Gamma}_n$
is the set of all connected graphs 
on $n$ labeled vertices $v_1,\ldots,v_n$, with multiple edges allowed but no loops 
(i.e. no edges connecting a vertex to itself). Both Equations~\eqref{eq:Hlongexpr0} and~\eqref{eq:Hlongexpr} hold for $n>1$ and $(g,n)\ne(0,2)$.

\section{General formula} \label{sec:MainTh}

In this section we prove the main theorem of the present paper, which explicitly represents $\Hc_{g,n}$ for given $g$ and $n$ in a closed form. What remains is to get rid of $D_1\cdots D_n$ which are applied in the LHS in Theorem~\ref{prop:mainprop}.

Let us introduce in a formal way the operator $D_{i}^{-1}U_i$ acting on a function $f(u_i,z_i)$ by
\begin{equation}\label{eq:DinvU}
D_i^{-1}U_{i} f=\sum_{j,r=0}^\infty D_i^{j-1}\left(\frac{[v^j]L_r(v,y(z_i),\hbar)}{Q_i}
[u_i^r] \frac{e^{u_i(\cS(u_i\,\hbar\,Q_i\,D_i)-1)y(z_i)}}{u_i\hbar\,\cS(u_i\,\hbar)}f(u_i,z_i)\right).
\end{equation}
where we \emph{define} the action of $D_i^{-1}$ on a function $w(z_i)$ by
\begin{equation}\label{eq:Dinv}
(D_i^{-1}w)(z_i)=\int_0^{z_i}\frac{Q(z)}{z}w(z)\,dz.
\end{equation}
Note that this formal definition of the operator $D_{i}^{-1}U_i$ implies that
\begin{equation}\label{eq:DDinv}
D_i\left(D_{i}^{-1}U_i f\right)=U_if.
\end{equation}
Then we set
\begin{equation}
\label{GenForm-graph}
\widetilde H_{g,n} = [\hbar^{2g-2+n}]  \left(\textstyle\prod_{i=1}^n D_i^{-1}U_i\right) \sum\limits_{\gamma\in\Gamma_n}\prod\limits_{\{v_i,v_j\}\in E_\gamma} w_{i,j},
\end{equation}
where $w_{i,j}$ is given by~\eqref{eq:wkl}.

Now we formulate the propositions needed to prove the theorem; their proofs are given below at the end of this section.

\begin{proposition}\label{prop:tildeH}
Assume that $n\ge2$ and $(g,n)\ne(0,2)$. Then each time when the operator $D_i^{-1}$ defined by~\eqref{eq:DinvU} is applied in the expression~\eqref{GenForm-graph} for $\widetilde H_{g,n}$ the corresponding integrated differential form $\frac{Q(z)}{z}w(z)\,dz$ is rational in~$z$ with possible poles at~$z=z_j$ for $j\ne i$ with zero residues. It follows that its primitive is well defined as a rational function in $z_i$ and the whole function $\widetilde H_{g,n}$ is well defined and has the form as in Theorem~\ref{th:intro} up to an additive constant.
\end{proposition}

By construction, we have $D_1\dots D_n H_{g,n}=D_1\dots D_n \widetilde H_{g,n}$. This does not imply that the functions $H_{g,n}$ and $\widetilde H_{g,n}$ are equal.

\begin{proposition}\label{prop:tHconst}
	For $n\geq 2$ and $(g,n)\ne(0,2)$, the difference between $H_{g,n}$ and $\widetilde{H}_{g,n}$ is the following constant:
	\begin{equation}\label{eq:tHconst}
	H_{g,n} = \widetilde{H}_{g,n} - (-1)^{n-1}
\psi^{(2g+n-2)}(0)[u^{2g}]\frac1{\cS^2(u)}.
	\end{equation}
\end{proposition}

Proposition \ref{prop:tHconst} directly implies the main theorem: 
\begin{theorem}\label{th:MainTh}
	For $n\geq 2$ and $(g,n)\ne(0,2)$ we have:
\begin{equation}\label{eq:mainth}
	\Hc_{g,n}= [\hbar^{2g-2+n}]  \left(\textstyle\prod_{i=1}^n D_i^{-1}U_i\right) \sum\limits_{\gamma\in\Gamma_n}\prod\limits_{\{v_i,v_j\}\in E_\gamma} w_{i,j}\\
+ (-1)^{n}\,\psi^{(2g+n-2)}(0)\;[u^{2g}]\frac1{\cS^2(u)},
\end{equation}
where $\Gamma_n$ is the set of simple graphs on $n$ vertices $v_1,\ldots,v_n$ with edges $E_\gamma$, $w_{i,j}$ is  given by~\eqref{eq:wkl}, and~$D_i^{-1}U_i$ is given by \eqref{eq:DinvU}--\eqref{eq:Dinv}.

	For fixed $g$ and $n$, after taking the coefficient $[\hbar^{2g-2+n}]$, 
	this formula turns into a rational expression in $z_1,\ldots,z_n$ and the derivatives of the functions $y_i=y(z_i)$ and $\psi(y_i)$.
\end{theorem}
\begin{remark}
Note that the structure of the obtained answer agrees with that suggested by Theorem~\ref{th:intro}. Thus, we have proved Theorem~\ref{th:intro} in the case $n\ge2$, $(g,n)\ne(0,2)$.
The special cases $n=1$ and $(g,n)=(0,2)$ are treated in the next section.
\end{remark}
\begin{remark}\label{rem:generalExplicit}
	Let us also provide 
	another
	form of the statement of the main theorem (narrowing it slightly to $n\geq 3$),
	where all integrals \eqref{eq:Dinv} are taken explicitly.
	Namely, for $n\geq 3$ we have:
	\begin{multline}
	\label{eq:explicitanswer}
	\Hc_{g,n}= 
	[\hbar^{2g-2+n}] \sum\limits_{\gamma\in\Gamma_n} \prod\limits_{v_i\in\mathcal{I}_\gamma} \overline{U}_i \prod\limits_{\{v_i,v_k\}\in E_\gamma\setminus\mathcal K_\gamma}w_{i,k}
	\times\\
	\prod\limits_{\{v_i,v_k\}\in\mathcal{K}_\gamma}\left( \overline{U}_i w_{i,k} +  \hbar u_k \cS(u_k\hbar Q_kD_k)\frac{z_i}{z_k-z_i}\right)
	+ (-1)^{n}\,\psi^{(2g+n-2)}(0)\;[u^{2g}]\frac1{\cS^2(u)},
	\end{multline}
	where $\Gamma_n$ is the set of simple graphs on $n$ vertices $v_1,\ldots,v_n$, $E_\gamma$ is the set of edges of a graph $\gamma$, $\mathcal{I}_\gamma$ is the subset of vertices which are not leaves and $\mathcal{K}_\gamma$ is the subset of edges with one end $v_i$ of valency $1$ and another end $v_k$, and where 
	\begin{align}	
	\overline U_i\, f &=  \sum_{r=0}^\infty\sum_{j=1}^\infty D_i^{j-1}\left(\frac{[v^j]L_r(v,y(z_i),\hbar)}{Q_i}
	[u_i^r] \frac{e^{u_i(\cS(u_i\,\hbar\,Q_i\,D_i)-1)y(z_i)}}{u_i\hbar\,\cS(u_i\,\hbar)}f\right),\\
	w_{k,\ell}&\mathop{=}^{\eqref{eq:wkl}}e^{\hbar^2u_ku_\ell\cS(u_k\hbar\,Q_k D_k)\cS(u_\ell\hbar\,Q_\ell D_\ell)
		\frac{z_k z_\ell}{(z_k-z_\ell)^2}}-1, \\
	D_i &\mathop{=}^{\eqref{eq:Ddef}} \dfrac{1}{Q_i}z \dfrac{\partial}{\partial z},\\
	L_r &\mathop{=}^{\eqref{eq:Lrdef}} \left(\partial_y+v\psi'(y)\right)^r e^{v\left(\frac{\cS(v\hbar\partial_y)}{\cS(\hbar\partial_y)}-1\right)\psi(y)},\\
	Q_i&\mathop{=}^{\eqref{eq:Qdef}} 1-z\,y'(z)\,\psi'(y(z)),\\
	\mathcal{S}(u)&\mathop{=}^{\eqref{eq:Sdef}}\dfrac{e^{u/2}-e^{-u/2}}{u}.
	\end{align}
	(to help the reader, we included in this list the notation and definitions of some functions introduced earlier in the paper). For $n=2$, $g>0$ the form of statement \eqref{eq:mainth} analogous to \eqref{eq:explicitanswer} is obtained in Section~\ref{sec:g2}. For brevity we do not provide the proof of the general case \eqref{eq:explicitanswer} in this text, but it is rather similar to the $(g,2)$ case of Section~\ref{sec:g2}.
\end{remark}

\begin{remark} \label{rem:Hdeformed}
Note that a statement similar to the statement of Theorem~\ref{th:MainTh}, as in the case of Theorem~\ref{prop:mainprop}, still holds if one allows $\psi(z)$ and $y(z)$ to also be formal series in $\hbar^2$. More precisely, it still holds in a very similar form, if one puts
\begin{align}
\psi(\hbar^2,y)&:=\sum_{k=1}^\infty\sum_{m=0}^\infty c_{k,m} y^k\hbar^{2m},\\
y(\hbar^2,z)&:=\sum_{k=1}^\infty\sum_{m=0}^\infty s_{k,m} z^k\hbar^{2m},
\end{align}
while still keeping the formula for $X(z)$ free of $\hbar$, i.e. using
\begin{equation}
X(z)=\left.z\,e^{-\psi(y(z))}\right|_{\hbar=0}
\end{equation}
in place of \eqref{eq:X(z)}, see~\cite[Section 3]{BDKS21toprec}.

\end{remark}

Analogously to the case of Theorem~\ref{prop:mainprop} and Corollary~\ref{cor:diagonal}, Theorem~\ref{th:MainTh} has a similar corollary (which is proved via exactly the same reasoning):
\begin{corollary}\label{cor:Hdiag}	
	All diagonal poles (i.e. poles at $z_i=z_j$ for $i\neq j$) on the right hand side of \eqref{eq:mainth} (and of \eqref{eq:explicitanswer})  cancel out.
\end{corollary}

Now we provide the proofs of the propositions of the present section.
\begin{proof}[Proof of Proposition~\ref{prop:tildeH}]
The operator $D_{i}^{-1}$ appears in the summand with $j=0$ in the definition of $D_i^{-1}U_i$. In the case $j=0$ we have $[v^0]L_0(v,y,\hbar)=1$ and $[v^0]L_r(v,y,\hbar)=0$ for $r>0$. Therefore, the summand with $j=0$ in~\eqref{eq:DinvU} can be written as
\begin{equation}\label{eq:contribDinv}
D_i^{-1}\frac{1}{Q_i}
[u_i^0] \frac{e^{u_i(\cS(u_i\,\hbar\,Q_i\,D_i)-1)y(z_i)}}{u_i\hbar\,\cS(u_i\,\hbar)}f=
D_{i}^{-1}\frac{1}{\hbar Q_i}[u_i^1]f.
\end{equation}

Let us check for which graphs~$\gamma$ the product $\prod\limits_{\{v_i,v_j\}\in E_\gamma} w_{i,j}$ has a non-vanishing linear term in $u_i$. By definition, $w_{i,j}$ is divisible by $u_iu_j$. It follows that if the vertex $v_i$ has valency greater that~$1$ then the contribution of such graph to the sum has vanishing linear term in~$u_i$.

If the vertex $i$ has valency 1 and is connected to the vertex $k$ then up to a factor that does not depend on $z_i$ the linear term in $u_i$ is the following:
\begin{equation}\label{eq:cleaf}
[u_i^1]w_{i,k} = \hbar^2 u_k \cS(u_k\hbar Q_kD_k)\frac{z_iz_k}{(z_i-z_k)^2}.
\end{equation}
The contribution of this term to~\eqref{eq:contribDinv} is given by
\begin{align}
D_{i}^{-1}\frac{1}{\hbar Q_i}[u_i^1]w_{i,k} \label{eq:leafexpr}
&=D_{i}^{-1}\frac{1}{Q_i}\hbar u_k \cS(u_k\hbar Q_kD_k)\frac{z_iz_k}{(z_i-z_k)^2}\\ \nonumber
&=\hbar u_k \cS(u_k\hbar Q_kD_k)z_k\int_0^{z_i}\frac{dz}{(z-z_k)^2}
\\ \notag
& =\hbar u_k \cS(u_k\hbar Q_kD_k)\frac{z_i}{z_k-z_i}.
\end{align}
This function is rational in $z_i$, as required.

This proves Proposition~\ref{prop:tildeH} in the case $n>2$. Indeed, we assumed implicitly in the above arguments that the leaf $v_i$ is connected to a vertex $v_k$ which is not a leaf so that $D_i^{-1}$ and $D_k^{-1}$ are not applied simultaneously. This is always the case for a connected graph with the number of vertices $n>2$. If $n=2$ then there could be summands linear both in $u_1$ and $u_2$ but these summands contribute to the case $g=0$ only. Therefore, the conclusion of Proposition holds in the case $n=2$ as well if $g>0$.
\end{proof}

\begin{proof}[Proof of Proposition~\ref{prop:tHconst}]
We regard all considered functions as elements of the ring\\ $\mathbb{C}[[z_1,\dots,z_n]][\{(z_i-z_j)^{-1};i,j=1,\ldots,n\}]$.
Let us denote by $I$ the ideal $(z_1\cdot\ldots\cdot z_n)$ generated by the product of coordinate functions.
$H_{g,n}$ itself lies in $I$, and we have, by construction (from Equations~\eqref{eq:mainprop}, \eqref{GenForm-graph}, and~\eqref{eq:DDinv}),
\begin{equation}
D_1\dots D_n H_{g,n}=D_1\dots D_n \widetilde H_{g,n}.
\end{equation}
Therefore, it suffices to show that the right hand side of~\eqref{eq:tHconst} belongs to~$I$.
Let us compute $\widetilde{H}_{g,n}$ modulo $I$.

Let $n\geq 3$.

From the proof of Proposition \ref{prop:tildeH} it follows that each internal edge $\{v_i,v_j\}$ of a graph $\gamma$ in the sum \eqref{GenForm-graph} brings a factor of $z_i z_j$. Indeed, $w_{i,j}$ itself belongs to the ideal $(z_iz_j)$, and the $j=0$ terms in the sums \eqref{eq:DinvU} for $D^{-1}U_i$ and $D^{-1}U_j$ vanish, while the $j>0$ terms cannot affect the property of divisibility by $z_i z_j$.

On the other hand, if $v_i$ is a leaf (connected to some $v_k$ of valence greater than 1),
then $\{v_i,v_k\}\in E_\gamma$ brings a factor of $z_i$, since, as above, $w_{i,k}$ is divisible by $z_iz_k$ and the $j>0$ terms of \eqref{eq:DinvU} cannot affect this property, while the $j=0$ term takes the form~$\hbar u_k \cS(u_k\hbar Q_kD_k)\frac{z_i}{z_k-z_i}$ (from Equation~\eqref{eq:leafexpr}), which is divisible by $z_i$.

Note that since $n\geq 3$ and the graphs are connected all edges belong to one of the above two cases.

Thus the contribution of the whole graph $\gamma$ is not divisible by $z_k$ for some $k$ only if the vertex $k$ is internal and all adjacent vertices are leaves. In this case the graph is the star with one vertex (labeled by $k$) of valency $n-1\geq 2$ and $n-1$ vertices of valency~$1$. We conclude that the contribution of all but the star graphs belong to $I$. The star graphs produce the following contributions:
\begin{align} 
	&\widetilde{H}_{g,n}+I =[\hbar^{2g-2+n}]\sum\limits_{k=1}^n D_k^{-1} U_k\prod_{i\ne k} D_i^{-1}U_i\prod_{i\ne k}w_{i,k}+I\\ \nonumber
	&=[\hbar^{2g-2+n}]\sum\limits_{k=1}^n \sum_{j,r=0}^\infty D_k^{j-1}\frac{[v^j]L_r(v,y(z_k),\hbar)}{Q_k}
	[u_k^r] \frac{e^{u_k(\cS(u_k\,\hbar\,Q_k\,D_k)-1)y(z_k)}}{u_k\hbar\,\cS(u_k\,\hbar)}
	\prod_{i\ne k} D_i^{-1}U_i\,w_{i,k}\;+\;I.
	\end{align}
Note that all $j=0$ terms vanish since $v_k$ is an internal vertex (with valence $\geq 2$), as discussed in the proof of Proposition \ref{prop:tildeH}. Thus we have
\begin{align} 
\widetilde{H}_{g,n}+I =[\hbar^{2g-2+n}]\sum\limits_{k=1}^n \sum_{\substack{j\geq 1 \\ r \geq 0}} D_k^{j-1}\frac{[v^j]L_r(v,y(z_k),\hbar)}{Q_k}
[u_k^r] \frac{e^{u_k(\cS(u_k\,\hbar\,Q_k\,D_k)-1)y(z_k)}}{u_k\hbar\,\cS(u_k\,\hbar)}
\\ \nonumber \times
\prod_{i\ne k} \sum_{j_i,r_i=0}^\infty D_i^{j_i-1}\left(\frac{[v^{j_i}]L_{r_i}(v,y(z_i),\hbar)}{Q_i}
[u_i^{r_i}] \frac{e^{u_i(\cS(u_i\,\hbar\,Q_i\,D_i)-1)y(z_i)}}{u_i\hbar\,\cS(u_i\,\hbar)}w_{i,k}\right)\;+\;I.
\end{align}
Now note that if any of $j_i>0$ then the corresponding term is divisible by $z_k$ since $w_{i,k}$ is divisible by $z_k$ and it gets acted upon only by operators of the sort $D_k^m$ and $D_i^m$ for $m\geq 0$ which do not spoil this property. Thus, we can factor out all these terms and we get, applying also formula \eqref{eq:leafexpr},
	\begin{align} \label{eq:HtildeI}
	\widetilde{H}_{g,n}+I & =[\hbar^{2g-2+n}]\sum\limits_{k=1}^n \sum_{\substack{j\geq 1 \\ r \geq 0}} D_k^{j-1}\frac{[v^j]L_r(v,y(z_k),\hbar)}{Q_k}
	[u_k^r] \frac{e^{u_k(\cS(u_k\,\hbar\,Q_k\,D_k)-1)y(z_k)}}{u_k\hbar\,\cS(u_k\,\hbar)}
	\\ \nonumber & \qquad \times
	\prod_{i\ne k}
\hbar u_k \cS(u_k\hbar Q_kD_k)\frac{z_i}{z_k-z_i}\;+\;I.
\end{align}
Now we note that all summands with $j\geq 2$ are also divisible by $z_k$ since $D_k=\frac{1}{Q_k}\,z_k\,\frac{\partial}{\partial z_k}$ and thus only the $j=1$ term remains.
Also note that $Q_k\equiv \cS(u_k\hbar Q_kD_k)\equiv 1\bmod {(z_k)}$. With the help of this, we obtain
\begin{equation}
	\widetilde{H}_{g,n}+I = [\hbar^{2g+n-2}]\sum\limits_{k=1}^n\sum\limits_{r=0}^\infty [v^1]L_r(v,y_k,\hbar)\;[u_k^r] \frac{1}{u_k\hbar \cS(u_k\hbar)}\prod\limits_{i\neq k}u_k\hbar\frac{z_i}{z_k-z_i} + I.
	\end{equation}
We have	
	\begin{align}\label{eq:L1expr}
	[v^1]L_r(v,y_k,\hbar)&\equiv[v^1](\partial_{y_k}+v\psi'(y_k))^r\left(1+v\left(\frac1{\cS(\hbar\partial_{y_k})}-1\right)\psi(y_k)\right)\bigm|_{y_k=0}  \\ \nonumber
	&\equiv\frac{\partial_{y_k}^r}{\cS(\hbar \partial_{y_k})}\psi(y_k)\bigm|_{y_k=0} \bmod{(z_k)}.
	\end{align}
	
	Using the fact
	\begin{equation}
		\sum\limits_{k=1}^n\prod\limits_{i\neq k} \frac{z_i}{z_k-z_i}=(-1)^{n-1}
	\end{equation}
 we finally obtain:	
	\begin{equation}
	\widetilde{H}_{g,n}+I=(-1)^{n-1}[\hbar^{2g+n-2}]\sum\limits_{r=0}^\infty \frac{\partial_{y}^r}{\cS(\hbar \partial_y)}\psi(y)\bigm|_{y=0} [u^r]\frac{(u\hbar)^{n-2}}{\cS(u\hbar)} + I.
	\end{equation}
	Note that we can reexpand the last sum in $\hbar$:
	\begin{align}\label{eq:psidertr}
	&(-1)^{n-1}\sum\limits_{r=0}^\infty \frac{\partial_{y}^r}{\cS(\hbar \partial_y)}\psi(y)\bigm|_{y=0} [u^r]\frac{(u\hbar)^{n-2}}{\cS(u\hbar)}=(-1)^{n-1}\frac1{\cS(\hbar \partial_y)}\frac{(\hbar\partial_y)^{n-2}}{\cS(\hbar \partial_y)}\psi(y)\bigm|_{y=0} \\ \nonumber
	& = (-1)^{n-1}\hbar^{n-2}\frac1{\cS^2(\hbar \partial_y)}\psi^{(n-2)}(y)\bigm|_{y=0}
	=(-1)^{n-1}\sum_{g=0}^\infty\hbar^{2g+n-2}\psi^{(2g+n-2)}(0)[u^{2g}]\frac1{\cS^2(u)},
	\end{align}
	thus
	\begin{equation}\label{eq:Htildeansw}
	\widetilde{H}_{g,n}+I=(-1)^{n-1}\psi^{(2g+n-2)}(0)[u^{2g}]\frac1{\cS^2(u)} + I.
	\end{equation}
	This concludes the proof for the $n\geq 3$ case.
	
For $n=2$, $g>0$ we have only one graph:
\begin{align}\label{eq:Hg2DD}
\widetilde{H}_{g,2}+I & =[\hbar^{2g}]D_1^{-1} U_1\, D_2^{-1} U_2\; w_{1,2} +I \\ \nonumber
&=[\hbar^{2g}]\sum_{j_1,r_1=0}^\infty D_1^{j_1-1}\frac{[v^{j_1}]L_{r_1}(v,y(z_1),\hbar)}{Q_1}
[u_1^{r_1}] \frac{e^{u_1(\cS(u_1\,\hbar\,Q_1\,D_1)-1)y(z_1)}}{u_1\hbar\,\cS(u_1\,\hbar)}
\\ \nonumber 
&\qquad \times\sum_{j_2,r_2=0}^\infty D_2^{j_2-1}\frac{[v^{j_2}]L_{r_2}(v,y(z_2),\hbar)}{Q_2}
[u_2^{r_2}] \frac{e^{u_2(\cS(u_2\,\hbar\,Q_2\,D_2)-1)y(z_2)}}{u_2\hbar\,\cS(u_2\,\hbar)}\;w_{1,2}\;+\;I.
\end{align}
Note that if both $j_1>0$ and $j_2>0$ then the corresponding terms are divisible by $z_1z_2$, analogous to what happened above. For $j_1=j_2=0$ we apply \eqref{eq:contribDinv} and get
\begin{align}
&[\hbar^{2g}]D_{1}^{-1}\frac{1}{\hbar Q_1}[u_1^1]D_{2}^{-1}\frac{1}{\hbar Q_2}[u_2^1] w_{1,2}\\ \nonumber &\quad= [\hbar^{2g}]D_{1}^{-1}\frac{1}{\hbar Q_1}[u_1^1]D_{2}^{-1}\frac{1}{\hbar Q_2}[u_2^1] \left(e^{\hbar^2u_1u_2\cS(u_1\hbar\,Q_1 D_1)\cS(u_2\hbar\,Q_2 D_2) \frac{z_1 z_2}{(z_1-z_2)^2}}-1\right)\\ \nonumber
&\quad= [\hbar^{2g}]D_{1}^{-1}\frac{1}{\hbar Q_1}D_{2}^{-1}\frac{1}{\hbar Q_2} \hbar^2\frac{z_1 z_2}{(z_1-z_2)^2},
\end{align}
which clearly vanishes for $g>0$.

Thus the sum in \eqref{eq:Hg2DD} can be represented as combination of two sums, one for $j_1=0,j_2>0$ and the other for $j_1>0, j_2=0$. This is actually precisely formula \eqref{eq:HtildeI} where one substitutes $n=2$. Thus we have reduced this case to the general $n$ case, so formula \eqref{eq:Htildeansw} holds here as well.

	
 This completes the proof of Proposition~\ref{prop:tHconst}.
\end{proof}
\begin{proof}[Proof of Theorem~\ref{th:MainTh}]
	The proof of the main statement immediately follows from Proposition~\ref{prop:tHconst}, while the rationality statement is implied by the respective rationality statement of Theorem \ref{prop:mainprop}.
\end{proof}

\section{Exceptional cases}\label{sec:exceptions}

Let us remind the definition of the functions $L_r$:
\begin{align}
L_0(v,y,\hbar)&:=e^{v\left(\frac{\cS(v\hbar\partial_y)}{\cS(\hbar\partial_y)}-1\right)\psi(y)},\\
L_r(v,y,\hbar)&:=e^{-v\psi(y)}\partial_y^re^{v\psi(y)}L_0(v,y,\hbar)
=\left(\partial_y+v\psi'(y)\right)^rL_0(v,y,\hbar).
\end{align}
In order to simplify the notation, we denote in computations of this section
\begin{equation}
L_{r,i}^j=[v^j]L_r(v,y(z_i),\hbar).
\end{equation}
Note that in the case $j=0$ we have
\begin{equation}
L_{0,i}^0=1,\qquad L_{r,i}^0=0\quad (r>0).
\end{equation}

\subsection{Computation of the \texorpdfstring{$(0,1)$}{(0,1)}-term}

Extracting the terms with $g=0$ in~\eqref{eq:Hnconn} for $n=1$ and noting $\phi_m(y)\bigm|_{\hbar=0}=e^{m\,\psi(y)}$ and $\cS(\hbar u)\bigm|_{\hbar=0}=1$ we get
\begin{equation}
D_1H_{0,1}=[\hbar^{-1}]D_1U_1^+\;1=\sum_{m=1}^\infty X_1^m\sum_{r=0}^\infty\partial_y^r e^{m\psi(y)}\bigm|_{y=0}[z^mu^r]
\frac{e^{u\,y(z)}}{u}.
\end{equation}
In order to apply Lemma~\ref{lem:L1} to the right hand side one needs to get rid of a pole in $u$ at the origin. One of the possibilities to do that is to differentiate this expression:
\begin{align}
D_1^2H_{0,1}&=\sum_{m=1}^\infty m X_1^m\sum_{r=0}^\infty\partial_y^r e^{m\psi(y)}\bigm|_{y=0}[z^mu^r]
\frac{e^{u\,y(z)}}{u}
\\ \notag &=
\sum_{m=1}^\infty X_1^m\sum_{r=0}^\infty\partial_y^r e^{m\psi(y)}\bigm|_{y=0}[z^mu^r]
z\partial_z\frac{e^{u\,y(z)}}{u}
\\ \notag &=
\sum_{m=1}^\infty X_1^m\sum_{r=0}^\infty\partial_y^r e^{m\psi(y)}\bigm|_{y=0}[z^mu^r]e^{u\,y(z)}QDy(z)
\\ \notag &=
\sum_{m=1}^\infty X_1^m[z^{m}]\sum_{r=0}^\infty
\partial_y^r e^{m\psi(y)}\bigm|_{y=0}\frac{y(z)^r}{r!}QDy(z)
\\ \notag &\xlongequal{\!'\!}
\sum_{m=1}^\infty X_1^m[z^{m}] e^{m\psi(y(z))}\,QDy(z)
\\ \notag &\xlongequal{\!''\!}
\sum_{m=-\infty}^\infty X_1^m[z^{m}] e^{m\psi(y(z))}\,QDy(z)
\\ \notag &\mathop{=}^{\eqref{eq:LagBur1}}
D_1y(z_1).
\end{align}
The equality $\xlongequal{\!'\!}$ is the Taylor series expansion, and the equality $\xlongequal{\!''\!}$ we obtain from the fact that for all $m\in \mathbb{Z}_{\leq 0}$ holds $[z^{m}] e^{m\psi(y(z))}\,QDy(z) =0$. 
The constant term equals to zero (after one performs integration of the equality $D_1D_1H_{0,1}=D_1 y(z_1)$) since both $D_1H_{0,1}$ and $y(z_1)$ are divisible by $z_1$. This proves equality
\begin{equation}	
	D_1H_{0,1}=y(z_1)
\end{equation} of Theorem~\ref{th:intro}.

\subsection{Computation of the \texorpdfstring{$(g,1)$}{(g,1)}-term, \texorpdfstring{$g>0$}{g>0}}

By~\eqref{eq:Hnconn}, we have

\begin{align}\label{eq:DHg1step1}
\hbar D_1H_{1}
&=\sum_{m=1}^\infty X_1^m\sum_{r=0}^\infty
\partial_y^r\phi_m(y,\hbar)\bigm|_{y=0}[z^{m}u^r]
\frac{e^{u\,\cS(u\,\hbar\,Q D)y(z)}}{u\,\cS(u\,\hbar)}\\ \notag
&=\sum_{m=1}^\infty X_1^m\sum_{r=0}^\infty
\partial_y^r\phi_m(y,\hbar)\bigm|_{y=0}[z^{m}u^r]\Bigl(
\frac{e^{u\,\cS(u\,\hbar\,Q D)y(z)}}{u\,\cS(u\,\hbar)}-
\frac{e^{u\,y(z)}}{u}\Bigr)
\\ \notag &\qquad
+\sum_{m=1}^\infty X_1^m\sum_{r=0}^\infty
\partial_y^r\phi_m(y,\hbar)\bigm|_{y=0}[z^{m}u^r]
\frac{e^{u\,y(z)}}{u}.
\end{align}
Expression of the first summand is regular in~$u$ and we can apply principal identity~\eqref{eq:princid} to get
\begin{align}
\label{eq:g1}
&\sum_{m=1}^\infty X_1^m\sum_{r=0}^\infty
\partial_y^r\phi_m(y,\hbar)\bigm|_{y=0}[z^{m}u^r]\Bigl(
\frac{e^{u\,\cS(u\,\hbar\,Q D)y(z)}}{u\,\cS(u\,\hbar)}-
\frac{e^{u\,y(z)}}{u}\Bigr)\\ \nonumber
&\qquad=\sum_{m=1}^\infty X_1^m\sum_{r=0}^\infty
\partial_y^r\phi_m(y,\hbar)\bigm|_{y=0}[z^{m}u^r]e^{uy}\Bigl(
\frac{e^{u\,(\cS(u\,\hbar\,Q D)-1)y(z)}}{u\,\cS(u\,\hbar)}-
\frac{1}{u}\Bigr)\\ \nonumber
&\qquad=\sum_{m=-\infty}^\infty X_1^m\sum_{r=0}^\infty
\partial_y^r\phi_m(y,\hbar)\bigm|_{y=0}[z^{m}u^r]e^{uy}\Bigl(
\frac{e^{u\,(\cS(u\,\hbar\,Q D)-1)y(z)}}{u\,\cS(u\,\hbar)}-
\frac{1}{u}\Bigr)\\ \nonumber
&\qquad\mathop{=}^{\eqref{eq:princid}}
\sum_{j,r=0}^\infty D_1^j\left(\frac{L_{r,1}^j}{Q_1}
[u^r] \left(\frac{e^{u\,(\cS(u\,\hbar\,Q D)-1)y(z_1)}}{u\,\cS(u\,\hbar)}-\dfrac{1}{u}\right)\right)\\ \nonumber
&\qquad=
\sum_{j,r=0}^\infty D_1^j\left(\frac{L_{r,1}^j}{Q_1}
[u^r] \frac{e^{u\,(\cS(u\,\hbar\,Q D)-1)y(z_1)}}{u\,\cS(u\,\hbar)}\right).
\end{align}
In the second equality we used that $\phi_0=1$ from the definition \eqref{eq:phi0def} and the fact that the expression after $[z^mu^r]$ does not contain negative powers of $z$ (we will use this switch from summation over $m$ starting from $0$ to summation over $m$ starting at $-\infty$ again in what follows, where it is applicable, without further commenting on it).
In the last equality the $1/u$ term disappears since the sum goes only over nonnegative $r$.
Note that \eqref{eq:g1} can be obtained if we take formally the right hand side of~\eqref{eq:Hlongexpr0} for the case~$n=1$.

The second summand in the right hand side of~\eqref{eq:DHg1step1} can be computed by the differentiation trick similar to the case~$g=0$ above. We have:
\begin{equation}
\begin{aligned}
D_1\sum_{m=1}^\infty & X_1^m\sum_{r=0}^\infty
\partial_y^r\phi_m(y,\hbar)\bigm|_{y=0}[z^{m}u^r]
\frac{e^{u\,y(z)}}{u}\\
&=\sum_{m=1}^\infty X_1^m\sum_{r=0}^\infty
\partial_y^r\phi_m(y,\hbar)\bigm|_{y=0}[z^{m}u^r]
e^{u\,y(z)}QDy(z)
\\&\mathop{=}^{\eqref{eq:princid}}
\sum_{j,r=0}^\infty D_1^j\left(\frac{[v^j]L_r(v,y(z),\hbar)}{Q}[u^r] QDy(z)\right)
\\&=\sum_{j=0}^\infty D_1^j\left(L_{0,1}^j\;D_1y(z_1)\right)
\\&=D_1y(z_1)+D_1\sum_{j=1}^\infty D_1^{j-1}\left(L_{0,1}^j\; D_1y(z_1)\right).
\end{aligned}
\end{equation}

Putting together and using again that the constant of integration equals to zero we conclude
\begin{equation}\label{eq:D1H1}
 D_1(\hbar\,H_{1}-H_{0,1})=
\sum_{j=0}^\infty D_1^j\left(\sum_{r=0}^\infty\frac{L_{r,1}^j}{Q_1}
[u^r] \frac{e^{u\,(\cS(u\,\hbar\,Q D)-1)y(z_1)}}{u\,\cS(u\,\hbar)}+L_{0,1}^{j+1}\; D_1y(z_1)\right),
\end{equation}
i.e. for $g>0$ we have
\begin{equation}
	D_1 H_{g,1}=[\hbar^{2g}]
	\sum_{j=0}^\infty D_1^j\left(\sum_{r=0}^\infty\frac{L_{r,1}^j}{Q_1}
	[u^r] \frac{e^{u\,(\cS(u\,\hbar\,Q D)-1)y(z_1)}}{u\,\cS(u\,\hbar)}+L_{0,1}^{j+1}\; D_1y(z_1)\right).
\end{equation}

Our next step is to invert the operator $D_1$ on the right hand side of Equation~\ref{eq:D1H1}. Possible problems can appear in the case $j=0$ only. Observe that the summand with $r=0$ is vanishing, which implies that the first summand in the term $j=0$ is also vanishing. The second summand in the term with $j=0$ is equal to
\begin{align}
L_{0,1}^1\,D_1 y(z_1)
& =\left(\frac1{\cS(\hbar \partial_y)}-1\right)\psi(y)\Bigm|_{y=y(z_1)}D_1 y(z_1)
\\ \notag
& =D_1\sum_{k=1}^\infty[u^{2k}]\frac1{\cS(u \hbar)}\psi^{(2k-1)}(y(z_1)).
\end{align}

If we define for $g>0$
\begin{align}\label{eq:Htildeg1}
\widetilde H_{g,1}& :=[\hbar^{2g}]\sum_{j=1}^\infty D_1^{j-1}\Bigl(\sum_{r=1}^\infty\frac{L_{r,1}^j}{Q_1}
[u^r] \frac{e^{u\,(\cS(u\,\hbar\,Q D)-1)y(z_1)}}{u\,\cS(u\,\hbar)}+L_{0,1}^{j+1}\; D_1y(z_1)\Bigr)
\\ \notag
& \qquad +\Bigl([u^{2g}]\frac1{\cS(u)}\Bigr)  \psi^{(2g-1)}(y(z_1)),
\end{align}
then we have $D_1H_{g,1}-D_1\widetilde H_{g,1}=0$. This means that $H_{g,1}$ and $\widetilde H_{g,1}$ may differ only by a constant. To determine this constant let us put $z_1=0$ in \eqref{eq:Htildeg1}. The second term in the brackets in the first line vanishes, as well as all terms in the $j$-sum for $j>1$, and the exponential and $Q$ both turn into $1$.
Let
\begin{equation}
\dfrac{1}{\cS(x)}=1+\sum_{k=1}^\infty \sigma_k x^{2k}.
\end{equation}
The first line of \eqref{eq:Htildeg1} for $z_1=0$ turns into the following:
\begin{align}
&[\hbar^{2g}]\sum_{r=1}^\infty L_{r,1}^1\bigm|_{y=0}
[u^r] \frac{1}{u\,\cS(u\,\hbar)} \\ \notag
& \qquad \mathop{=}^{\eqref{eq:L1expr}}[\hbar^{2g}]\sum_{r=1}^\infty \frac{\partial_{y}^r}{\cS(\hbar \partial_{y})}\psi(y)\bigm|_{y=0} [u^r] \frac{1}{u\,\cS(u\,\hbar)} \\ \nonumber
&\qquad =[\hbar^{2g}]\sum_{r=1}^\infty \frac{\partial_{y}^r}{\cS(\hbar \partial_{y})}\psi(y)\bigm|_{y=0} [u^r] \left(\frac{1}{u\,\cS(u\,\hbar)}-\dfrac{1}{u}\right)\\ \nonumber
&\qquad =[\hbar^{2g}]\sum_{r=2}^\infty \frac{\partial_{y}^{r-1}}{\cS(\hbar \partial_{y})}\psi(y)\bigm|_{y=0} [u^{r}] \left(\frac{1}{\cS(u\,\hbar)}-1\right)\\ \nonumber
&\qquad =[\hbar^{2g}]\sum_{r=2}^\infty \partial_{y}^{r-1}\left(1+\sum_{k=1}^\infty \sigma_k \hbar^{2k}\partial_y^{2k}\right)\psi(y)\bigm|_{y=0} [u^{r}] \left(\sum_{m=1}^\infty \sigma_m \hbar^{2m}u^{2m}\right)\\  \notag
&\qquad =[\hbar^{2g}]\left(\sum_{m=1}^\infty \sigma_m \hbar^{2m}\partial_y^{2m-1}\right)\left(1+\sum_{k=1}^\infty \sigma_k \hbar^{2k}\partial_y^{2k}\right)\psi(y)\bigm|_{y=0}\\ \nonumber
&\qquad =\psi^{(2g-1)}(y)(0)\;[u^{2g}]\left(\dfrac{1}{\cS(u)^2}-\dfrac{1}{\cS(u)}\right).
\end{align}
Setting $z_1=0$ in the second line of \eqref{eq:Htildeg1} is trivial and we arrive at
\begin{proposition} For $n=1$ and $g>0$ we have:
\begin{align} \label{eq:(g,1)}
H_{g,1}& =[\hbar^{2g}]\sum_{j=1}^\infty D_1^{j-1}\Biggl(\sum_{r=1}^\infty\frac{L_{r,1}^j}{Q_1}
[u^r] \frac{e^{u\,(\cS(u\,\hbar\,Q D)-1)y(z_1)}}{u\,\cS(u\,\hbar)}+L_{0,1}^{j+1}\; D_1y(z_1)\Biggr)
\\ \notag & \quad
+\Bigl([u^{2g}]\frac1{\cS(u)}\Bigr)  \psi^{(2g-1)}(y(z_1))
-\Bigl([u^{2g}]\frac1{\cS(u)^2}\Bigr)  \psi^{(2g-1)}(0).
\end{align}
\end{proposition}

Note that the structure of this formula agrees with the statement of Theorem~\ref{th:intro}.

\subsection{Computation of the \texorpdfstring{$(0,2)$}{(0,2)}-term}

We have
\begin{align}
D_1D_2H_{0,2}+\frac{X_1X_2}{(X_1-X_2)^2}& \mathop{=}^{\eqref{eq:DH02}}\widehat{DH}_{0,2}
\\ \notag &
\mathop{=}^{\eqref{eq:hatDH}}[\hbar^0]U_2U_1w_{1,2}\\ \notag
&\mathop{=}^{\eqref{eq:principal}}
[\hbar^0]\sum_{j_1,j_2=0}^\infty D_1^{j_1}D_2^{j_2}\sum_{r_1,r_2=0}^\infty
\tfrac{L_{r_1,1}^{j_1}L_{r_2,2}^{j_2}}{Q_1Q_2}
[u_1^{r_1}u_2^{r_2}]\frac{z_1z_2}{(z_1-z_2)^2}\\ \notag
&=[\hbar^0]\sum_{j_1,j_2=0}^\infty D_1^{j_1}D_2^{j_2}
\tfrac{L_{0,1}^{j_1}L_{0,2}^{j_2}}{Q_1Q_2}
\frac{z_1z_2}{(z_1-z_2)^2}\\ \notag
&=\frac{1}{Q_1Q_2}\frac{z_1z_2}{(z_1-z_2)^2}.
\end{align}
Thus, we get
\begin{align}
D_1D_2H_{0,2}&=\frac{1}{Q_1Q_2}\frac{z_1z_2}{(z_1-z_2)^2}-\frac{X_1X_2}{(X_1-X_2)^2}\\ \notag
&=D_1\left(\frac{1}{Q_2}\frac{z_1}{z_2-z_1}-\frac{X_1}{X_2-X_1}\right)\\ \notag
&=D_1D_2\log\left(\frac {z_1^{-1}-z_2^{-1}}{X_1^{-1}-X_2^{-1}}\right).
\end{align}

The function $\widetilde H_{0,2}=\log\left(\frac {z_1^{-1}-z_2^{-1}}{X_1^{-1}-X_2^{-1}}\right)$ represents a regular series vanishing at $z_1=0$ and at $z_2=0$ and satisfies $D_1D_2H_{0,2}=D_1D_2\widetilde H_{0,2}$. Therefore, it coincides with $H_{0,2}$. This proves~\eqref{eq:H02}.

This completes the proof of remaining exceptional cases of Theorem~\ref{th:intro}.

\subsection{Computation of the \texorpdfstring{$(g,2)$}{(g,2)}-term, \texorpdfstring{$g>0$}{g>0}}\label{sec:g2}
This case is actually already covered by Theorem \ref{th:MainTh}, but we can present a more explicit form of the answer (in line with Remark~\ref{rem:generalExplicit}). We have
\begin{align}
D_1D_2H_{g,2}&\mathop{=}^{\eqref{eq:preDHgn}}[\hbar^{2g}]U_2U_1w_{1,2}\\ \notag
&\mathop{=}^{\eqref{eq:cleaf}}
[\hbar^{2g}]U_2
D_1 \left(\overline{U}_1 w_{1,2} + \hbar u_2 \cS(u_2\hbar Q_2D_2)\frac{z_1}{z_2-z_1}\right)
\\ \notag
&=
[\hbar^{2g}]
D_1 \left(\overline{U}_1 U_2 w_{1,2} + U_2\hbar u_2 \cS(u_2\hbar Q_2D_2)\frac{z_1}{z_2-z_1}\right)
\\ \notag &=
D_1 D_2\widetilde H_{g,2},
\end{align}
where
\begin{equation}
\widetilde H_{g,2}=[\hbar^{2g}]\left(\overline{U}_1\Bigl(\overline{U}_2 w_{1,2}
+ \hbar u_1 \cS(u_1\hbar Q_1D_1)\frac{z_2}{z_1-z_2}\Bigr)
+\overline{U}_2\Bigl( \hbar u_2 \cS(u_2\hbar Q_2D_2)\frac{z_1}{z_2-z_1}\Bigr)\right)
.
\end{equation}
One extra term that we omitted here contributes only in the case $g=0$, which we considered above.

Arguing as in the proof of Proposition~\ref{prop:tHconst}, we conclude that $H_{g,2}$ and $\widetilde H_{g,2}$ differ by a constant that is given by the same formula as in the general case, and we obtain

\begin{proposition} For $n=2$ and $g>0$ we have:
\begin{align}
H_{g,2}& =[\hbar^{2g}]\Biggl(\overline{U}_1\overline{U}_2 w_{1,2}
+ \overline{U}_1\Bigl(\hbar u_1 \cS(u_1\hbar Q_1D_1)\frac{z_2}{z_1-z_2}\Bigr)
\Biggr.\\ \notag \Biggl.
& \qquad +\overline{U}_2\Bigl( \hbar u_2 \cS(u_2\hbar Q_2D_2)\frac{z_1}{z_2-z_1}\Bigr)\Biggr)
	+ \psi^{(2g)}(0)\;[u^{2g}]\frac1{\cS^2(u)}.
\end{align}
\end{proposition}

Remark that the structure of the obtained answer agrees with that suggested by Theorem~\ref{th:intro} and correlates with Equation~\eqref{eq:explicitanswer}.

\section{Applying general formula} \label{sec:applying}

In this section we derive explicit expressions for $H_{g,n}$ for small $g$ and $n$ in terms of small number of basic functions. These functions include:
\begin{equation}
\begin{aligned}
\psi^{(k)}_i&=\psi^{(k)}(y(z_i)),\quad k\ge 1,\\
y^{[k]}_i&=\left(z_i\partial_{z_i}\right)^k y(z_i),\quad k\ge1,\\
Q_i&=Q(z_i)=1-\psi'_i y^{[1]}_i.
\end{aligned}
\end{equation}
If $n=1$ we set $z_1=z$ and drop the bottom index $i=1$. In the case $n\ge2$ we will use also functions
\begin{equation}
\begin{aligned}
\gamma_{i,j}&=\gamma_{j,i}=\frac{z_iz_j}{(z_i-z_j)^2},\\
\gamma^{[k]}_{i,j}&=(-1)^k\gamma^{[k]}_{j,i}=\left(z_i\partial_{z_i}\right)^k\gamma_{i,j},\quad k\ge0.
\end{aligned}
\end{equation}
Then, according to (the proof of) Proposition~\ref{prop:tildeH}, the application of $D_i^{-1}$ is reduced to
\begin{equation}
D_i^{-1}(Q_i \gamma^{[k]}_{i,j})=(-1)^kD_i^{-1}(Q_i \gamma^{[k]}_{j,i})=\left(z_i\partial_{z_i}\right)^{-1}\gamma^{[k]}_{i,j}:=\gamma^{[k-1]}_{i,j}.
\end{equation}
This formula can be applied also for $k=0$ if we set, in addition,
\begin{equation}
\gamma^{[-1]}_{i,j}=-1-\gamma^{[-1]}_{j,i}=\frac{z_i}{z_j-z_i}.
\end{equation}

\subsection{Computations for \texorpdfstring{$n=1$}{n=1}}

Substituting the genus expansions
\begin{equation}\label{eq:V-exp}
\tfrac{e^{u(\cS(u\,\hbar\, z\partial_{z})-1)y(z)}}{u\cS(u\hbar)}
=\frac{1}{u}+\frac{1}{24} \left(u^2 y^{[2]}-u \right)\hbar ^2 +O(\hbar^4)
\end{equation}
to Equation~\eqref{eq:(g,1)} in the case $n=1$, $g\ge1$, we obtain
\begin{align}
H_{g,1} & =[\hbar^{2g}]\sum_{j=0}^\infty D^j \tfrac{1}{Q}[v^j]
\left(\tfrac1{24}\left(\tfrac{L_2(v,y,\hbar)}{v}y^{[2]}-\tfrac{L_1(v,y,\hbar)}{v}\right)\hbar^2
+\tfrac{L_0(v,y,\hbar)}{v^2}+O(\hbar^4)\right)
\\ \notag & \qquad
+[u^{2g}]\tfrac1{\cS(u)}\;  \psi^{(2g-1)}(y(z))
-[u^{2g}]\tfrac1{\cS(u)^2}\;  \psi^{(2g-1)}(0).
\end{align}
Then, using explicit expressions for the series $L_r$,
\begin{align}\label{eq:L-exp}
L_0(v,y,\hbar)&=1+( v^3- v)\frac{\psi ''(y)}{24} \hbar ^2+O(\hbar^4),\\ \notag
L_1(v,y,\hbar)&=v\,\psi'(y)+O(\hbar^2),\\ \notag
L_2(v,y,\hbar)&=v\,\psi ''(y)+v^2 \psi '(y)^2+O(\hbar^2),\\ \notag
L_3(v,y,\hbar)&=v^3 \psi '(y)^3+3\,v^2 \psi '(y) \psi ''(y)+v\,\psi ^{(3)}(y)+O(\hbar^2),
\end{align}
we obtain, in the case $g=1$,
\begin{equation}
H_{1,1}
=D \tfrac{\left(\psi'\right)^2 y^{[2]}+\psi''y^{[1]}}{24Q}+
\tfrac{\psi'' y^{[2]}-\psi'}{24Q}-\tfrac{\psi'}{24}+\tfrac{\psi'(0)}{12}.
\end{equation}

Similar computations in the case $(g,n)=(2,1)$ give
\begin{align}
H_{2,1}(z)& =D^4\tfrac{10 \psi'' \left(\psi'\right)^2 y^{[2]}+5 \left(\psi''\right)^2 y^{[1]}+5 \left(\psi'\right)^5 \left(y^{[2]}\right)^2}{5760\,Q}
+\cdots\\ \notag & \qquad
+\tfrac{5 \psi^{(5)} \left(y^{[2]}\right)^2-20 \psi^{(4)} y^{[2]}+3 \psi^{(4)} y^{[4]}+17 \psi^{(3)}+5 \left(\psi''\right)^2 y^{[1]}}{5760\,Q}
+\tfrac{7\psi^{(3)}}{5760}-\tfrac{\psi ^{(3)}(0)}{240}
\end{align}
where the dots denote the terms containing $D^j$ with $j=1,2,3$.

\subsection{Computations for \texorpdfstring{$n=2$}{n=2}}
If $n\ge1$, then equation of Theorem~\ref{th:MainTh} can be applied.
It is convenient to represent the transformation $D^{-1}U$ of Theorem~\ref{th:MainTh} acting on a function $f(u,z)$ in~$u$ and $z$ as follows
\begin{equation}
D^{-1}U f=\frac1{\hbar}\sum_{r=0}^\infty M_r([u^r]f)
\end{equation}
where $M_r$ is the differential operator acting on a function $f(z)$ in $z$ by
\begin{equation}
M_rf=\sum_{k,j=0}^\infty D^{j-1}\left(\tfrac{[v^j] L_{k}(v,y(z),\hbar)}{Q}\;[u^k]u^r\tfrac{e^{u(\cS(u\,\hbar\,z\partial_z)-1)y(z)}}{u \cS(u\hbar)}\;f\right).
\end{equation}

From~\eqref{eq:V-exp} and~\eqref{eq:L-exp} we find, explicitly,
\begin{align}
M_1f&=D^{-1}\tfrac{f}{Q}+
\left(\tfrac{\left(\psi^{(3)} y^{[2]}-2 \psi''\right)f}{24Q}
+ D\tfrac{\psi' \left(3 \psi'' y^{[2]}-\psi'\right)f}{24Q}
+ D^2\tfrac{\left(\left(\psi'\right)^3 y^{[2]}+\psi''\right)f}{24Q}
\right)\hbar ^2 +O(\hbar^4),\\ \notag
M_2f&=\tfrac{\psi'f}{Q}+O(\hbar^2),\\ \notag
M_3f&=
\tfrac{\psi''f}{Q}
+D \tfrac{\left(\psi'\right)^2f}{Q}
+O(\hbar^2).
\end{align}

We denote by $M_{k,i}$ the transformation $M_k$ applied to the functions in~$u_i$ and~$z_i$ instead of~$u$ and~$z$, respectively. With this notation, the statement of Theorem~\ref{th:MainTh} can be written as follows
\begin{align}
H_{g,n} & = [\hbar^{2g}] \sum_{r_1,\dots,r_n=1}^\infty M_{r_1,1}\dots M_{r_n,n} [u_1^{r_1}\dots u_n^{r_n}]
\sum\limits_{\gamma\in\Gamma_n}\hbar^{2(|E_\gamma|-n+1)}\prod\limits_{\{v_i,v_j\}\in E_\gamma} \overline w_{i,j}\\ \notag & \qquad
+ (-1)^{n}[u^{2g}]\tfrac1{\cS^2(u)}\;\psi^{(2g+n-2)}(0),
\end{align}
where
\begin{align}
\overline w_{i,j}=\frac{w_{i,j}}{\hbar^2}& =\frac{e^{\hbar^2u_iu_j\cS(u_i\hbar z_i\partial_{z_i})\cS(u_j\hbar z_j\partial_{z_j})\gamma_{i,j}}-1}{\hbar^{2}}\\ \notag
&=u_iu_j\gamma_{i,j}+\left(\frac{u_i^3u_j+u_iu_j^3}{24}\gamma_{i,j}^{[2]}
+\frac12 u_i^2u_j^2(\gamma_{i,j})^2\right)\hbar^2+O(\hbar^2).
\end{align}

If $n=2$, then the sum over graphs is reduced to just $\overline w_{1,2}$ and we get
\begin{align}
H_{g,2}&= [\hbar^{2g}] \sum_{r_1,r_2=1}^\infty M_{r_1,1} M_{r_2,2}[u_1^{r_1}u_2^{r_2}]\overline w_{1,2}
+ [u^{2g}]\tfrac1{\cS^2(u)}\psi^{(2g)}(0)\\ \notag
&=
[\hbar^{2g}] \Biggl(M_{1,1}M_{1,2}\gamma_{1,2}
+\left(M_{3,1}M_{1,2}\frac{\gamma_{1,2}^{[2]}}{24}
+M_{1,1}M_{3,2}\frac{\gamma_{1,2}^{[2]}}{24}+M_{2,1}M_{2,2}\frac{(\gamma_{1,2})^2}{2}\right)\hbar^2\Biggr.
\\ \notag &\qquad\qquad\qquad \Biggl.{}+O(\hbar^4)\Biggr)
+ [u^{2g}]\tfrac1{\cS^2(u)}\psi^{(2g)}(0).
\end{align}
In particular, for $(g,n)=(1,2)$ we have
\begin{align}
H_{1,2}&=
D_1^2\tfrac{\gamma_{2,1}^{[-1]} (\psi_1''+(\psi_1')^3 y_1^{[2]})}{24Q_1}+
D_1\tfrac{\psi_1' (\psi_1' (\gamma_{2,1}^{[1]}-\gamma_{2,1}^{[-1]})+3 \gamma_{2,1}^{[-1]} \psi_1'' y_1^{[2]})}
{24Q_1}+
\tfrac{\psi_1''(\gamma_{2,1}^{[1]}-2 \gamma_{2,1}^{[-1]})+\psi_1^{(3)} \gamma_{2,1}^{[-1]} y_1^{[2]}}{24Q_1}
\\ \notag & \qquad
+D_2^2\tfrac{\gamma_{1,2}^{[-1]} (\psi_2''+(\psi_2')^3 y_2^{[2]})}{24Q_2}+
D_2\tfrac{\psi_2' (\psi_2' (\gamma_{1,2}^{[1]}-\gamma_{1,2}^{[-1]})+3 \gamma_{1,2}^{[-1]} \psi_2'' y_2^{[2]})}
{24Q_2}+
\tfrac{\psi_2''(\gamma_{1,2}^{[1]}-2 \gamma_{1,2}^{[-1]})+\psi_2^{(3)} \gamma_{1,2}^{[-1]} y_2^{[2]}}{24Q_2}
\\ \notag &\qquad
+\tfrac{(\gamma_{1,2})^2 \psi_1' \psi_2'}{2Q_1Q_2}
-\frac{1}{12}\psi''(0).
\end{align}

\subsection{Computations for \texorpdfstring{$n=3$}{n=3}}
There are four possible connected simple graphs on three labeled vertices, and summing up the contributions of these four graphs we get
\begin{align}
H_{g,3}& =[\hbar^{2g}]\sum_{r_1,r_2,r_3=1}^\infty M_{r_1,1} M_{r_2,2} M_{r_3,3}[u_1^{r_1}u_2^{r_2}u_3^{r_3}]
\bigl(\overline w_{1,2}\overline w_{1,3}+\overline w_{1,2}\overline w_{3,3}+\overline w_{1,3}\overline w_{2,3}
\\ \notag & \qquad +\hbar^2 \overline w_{1,2}\overline w_{1,3}\overline w_{2,3}\bigr)
- [u^{2g}]\tfrac1{\cS^2(u)}\psi^{(2g+1)}(0).
\end{align}
For instance, for $g=0$ using $\overline w_{i,j}=u_iu_j\gamma_{i,j}+O(\hbar^2)$ we get
\begin{equation}
H_{0,3}=\Bigl(
M_{2,1}M_{1,2}M_{1,3}\gamma_{1,2}\gamma_{1,3}+
M_{1,1}M_{2,2}M_{1,3}\gamma_{1,2}\gamma_{2,3}+
M_{1,1}M_{1,2}M_{2,3}\gamma_{1,3}\gamma_{2,2}\Bigr)\Bigm|_{\hbar=0}-\psi'(0).
\end{equation}
This gives the final answer
\begin{align}\label{eq:H03}
H_{0,3}&=
\frac{\psi_1'}{Q_1}\gamma_{2,1}^{[-1]}\gamma_{3,1}^{[-1]}
+\frac{\psi_2'}{Q_2}\gamma_{1,2}^{[-1]}\gamma_{3,2}^{[-1]}
+\frac{\psi_3'}{Q_3}\gamma_{1,3}^{[-1]}\gamma_{2,3}^{[-1]}
-\psi'(0)\\ \nonumber
&=
\sum_{i=1}^3\frac{\psi'(y_i)}{Q(z_i)}\prod_{j\ne i}\frac{z_j}{z_i-z_j}-\psi'(0).
\end{align}
\begin{remark} Note that Equation~\eqref{eq:H03} differs from \cite[Proposition 10.2 and Equation (10.4)]{ACEH-2} produced with the help of the spectral curve topological recursion. The formula given in Equation~(10.4) in~\emph{op.~cit.} does not appear to be vanishing on the coordinate axes and seems to have an incorrect overall sign, which are typical bugs that often occur in applications of topological recursion.
\end{remark}

\subsection{Computation for \texorpdfstring{$(g,n)=(0,4)$}{(g,n)=(0,4)}}
In the case $g=0$ the graphs that contribute to $H_{0,n}$ are trees. For $n=4$
there are
$4$ trees on $4$ labeled vertices isomorphic to%
\begin{picture}(25,5)(-4,0)
\put(0,0){\circle*{2}}
\put(0,8){\circle*{2}}
\put(10,4){\circle*{2}}
\put(18,4){\circle*{2}}
\put(0,0){\line(5,2){10}}
\put(0,8){\line(5,-2){10}}
\put(10,4){\line(1,0){8}}
\end{picture}
and $12$  more trees isomorphic to%
\begin{picture}(28,5)(-4,-3)
\put(0,0){\circle*{2}}
\put(7,0){\circle*{2}}
\put(14,0){\circle*{2}}
\put(21,0){\circle*{2}}
\put(0,0){\line(1,0){21}}
\end{picture}.
They contribute to the corresponding summands in~$H_{0,4}$:
\begin{align}
H_{0,4}&= [\hbar^0]\sum_{r_1,\dots,r_4\geq 1}
\left(\textstyle \prod_{k=1}^4 M_{r_k,k}\right)
\Bigl[{ \textstyle\prod_{k=1}^4 u_k^{r_k}}\Bigr]\;
\\ \notag &\qquad\qquad
\Biggl(
\bigl(u_1u_2\gamma_{1,2}u_1u_3\gamma_{1,3}u_1u_4\gamma_{1,4}+\dots(\text{\it$4$ terms in total})\bigr)
\\ \notag &\qquad \qquad
+\bigl(u_1u_2\gamma_{1,2}u_2u_3\gamma_{2,3}u_3u_4\gamma_{3,4}+\dots(\text{\it$12$ terms in total})\bigr)
\Biggr)
\\ \notag & \qquad +\psi''(0) \\ \notag
&=
\Bigl(
\Bigl(M_{3,1}M_{1,2}M_{1,3}M_{1,4}\bigl(\gamma_{1,2}\gamma_{1,3}\gamma_{1,4}\bigr)
+\dots(\text{\it$4$ terms in total})\Bigr)
\Biggr.\Biggr.\\ \notag &\qquad\Biggl.\Biggl.
+
\Bigl(M_{1,1}M_{2,2}M_{2,3}M_{1,4}\bigl(\gamma_{1,2}\gamma_{2,3}\gamma_{3,4}\bigr)
+\dots(\text{\it$12$ terms in total})\Bigr)
\Bigr)\Bigm|_{\hbar=0}+\;\psi''(0),
\end{align}
and we get the final answer
\begin{align}
H_{0,4}&=
\Bigl(D_1\frac{(\psi'_1)^2\gamma_{2,1}^{[-1]}\gamma_{3,1}^{[-1]}\gamma_{4,1}^{[-1]}}{Q_1}+
\frac{\psi''_1\gamma_{2,1}^{[-1]}\gamma_{3,1}^{[-1]}\gamma_{4,1}^{[-1]}}{Q_1}
+\dots(\text{\it$2\times4$ terms in total})\Bigr)
\\ \notag & \qquad +
\Bigl(\frac{\psi'_2\psi'_3\gamma_{1,2}^{[-1]}\gamma_{2,3}\gamma_{4,3}^{[-1]}}
{Q_2Q_3}+\dots(\text{\it$12$ terms in total})\Bigr)
+\psi''(0).
\end{align}

\bibliographystyle{alphaurl}
\bibliography{top_rec}

\end{document}